\documentclass{amsart}

\usepackage[utf8]{inputenc}
\usepackage{color}
\usepackage{xcolor}
\usepackage{hyperref}
\usepackage{amssymb}
\usepackage{mathabx}
\usepackage{float}
\usepackage{tikz}
\usepackage{tikz-cd}
\usepackage{forest}
\usepackage{nicefrac}
\usepackage{enumerate}
\usepackage[nameinlink, capitalize, noabbrev]{cleveref}
\usetikzlibrary{shapes.geometric,decorations.markings}
\usetikzlibrary{decorations.pathreplacing}
\usetikzlibrary{fit}
\usetikzlibrary{positioning}
\usetikzlibrary{intersections}
\tikzset{mytext/.style={font=\small, text=black}}

\hypersetup{
    colorlinks=true,
    linkcolor=teal,
    citecolor=magenta,
    }

\newtheorem{Theorem}{Theorem}

\newtheorem{Conjecture}{Conjecture}

\newtheorem{proposition}{Proposition}[section]
\newtheorem{lemma}[proposition]{Lemma}
\newtheorem{corollary}[proposition]{Corollary}
\newtheorem{theorem}[proposition]{Theorem}

\newtheorem{Question}[Conjecture]{Question}
\theoremstyle{definition}

\numberwithin{equation}{section}

\title{Restricted Hausdorff spectra of $q$-adic automorphisms}
\author{Jorge Fariña-Asategui}
\address{Jorge Fariña-Asategui: Centre for Mathematical Sciences, Lund University, 223 62 Lund, Sweden -- Department of Mathematics, University of the Basque Country UPV/EHU, 48080 Bilbao, Spain}
\email{jorge.farina\_asategui@math.lu.se}
\thanks{The author is supported by the Spanish Government, grant PID2020-117281GB-I00, partly with FEDER funds. The author also acknowledges support from the Walter Gyllenberg Foundation from the Royal Physiographic Society of Lund}

\begin{document}

\begin{abstract}

Firstly, we completely determine the self-similar Hausdorff spectrum of the group of $q$-adic automorphisms where $q$ is a prime power, answering a question of Grigorchuk. Indeed, we take a further step and completely determine its Hausdorff spectra restricted to the most important subclasses of self-similar groups, providing examples differing drastically from the previously known ones in the literature. Our proof relies on a new explicit formula for the computation of the Hausdorff dimension of closed self-similar groups and a generalization of iterated permutational wreath products. 

Secondly, we provide for every prime $p$ the first examples of just infinite branch pro-$p$ groups with zero Hausdorff dimension in $\Gamma_p$, giving strong evidence against a well-known conjecture of Boston.
\end{abstract}

\maketitle

\section{introduction}
\label{section: introduction}

The study of the Hausdorff dimension on profinite groups was initiated by Abercrombie \cite{Abercrombie} and by Barnea and Shalev \cite{BarneaShalev} in the 1990s, and it has generated a lot of interest recently; see \cite{PadicAnalytic} for an overview and see \cite{IkerBenjamin, IkerAnitha1, IkerAnitha2, OhianaAlejandraBenjamin, AndoniJon} for more recent work. For a countably-based profinite group $G$ endowed with a filtration of open normal subgroups $\{G_n\}_{n\ge 1}$, the Hausdorff dimension in $G$ of a closed subgroup $H\le G$ with respect to $\{G_n\}_{n\ge 1}$ is given by
$$\mathrm{hdim}_{G}^{\{G_n\}}(H)=\liminf_{n\to\infty}\frac{\log|HG_n:G_n|}{\log|G:G_n|}\in [0,1].$$
The Hausdorff dimension depends on the choice of the filtration $\{G_n\}_{n\ge 1}$. For profinite groups acting on rooted trees a natural congruence filtration is available, namely the one given by level stabilizers, whereas for general pro-$p$ groups there are several standard filtrations widely used in the literature; see \cite{PadicAnalytic}.

An important family of pro-$p$ groups is the one of $p$-adic analytic pro-$p$ groups, i.e. pro-$p$ groups which are analytic manifolds over the $p$-adic numbers and both the group multiplication and inversion are analytic functions; see \cite{PadicBook} for a good introduction to the topic. Barnea and Shalev \cite{BarneaShalev} proved that the Hausdorff dimension of closed subgroups of $p$-adic analytic pro-$p$ groups with respect to the filtration $\{G^{p^n}\}_{n\ge 1}$ coincides with their normalized dimension as Lie groups. Furthermore, they characterized $p$-adic analytic pro-$p$ groups in terms of the positivity of the Hausdorff dimension of their infinite closed subgroups based on a result of Zelmanov~\cite{Zelmanov}. They conjectured that $p$-adic analytic pro-$p$ groups can be characterized in terms of their \textit{Hausdorff spectrum}, i.e. by the set of Hausdorff dimensions of all their closed subgroups, as those pro-$p$ groups with finite Hausdorff spectrum. This conjecture remains open after more than two decades with only positive results for some particular cases \cite{PadicAnalytic}.

An important family of groups deserving interest is the one of \textit{branch} groups, i.e. those groups with a tree-like subnormal subgroup structure. Branch groups have been extensively studied for several decades due to their interesting algebraic properties, which have answered many open questions in group theory; see \cite{Handbook} for an introduction to branch groups. Just infinite branch profinite groups constitute one of the two disjoint classes of groups partitioning the class of \textit{just infinite} profinite groups, that is, profinite groups all of whose non-trivial closed subgroups are open \cite{NewHorizonsGrigorchuk, Wilson}. The other class is the one of \textit{hereditarily just infinite} profinite groups, those just infinite profinite groups whose non-trivial closed subgroups are again just infinite profinite.  Among branch profinite groups, of particular interest is the group of \textit{$p$-adic automorphisms}~$\Gamma_p$, as it contains every countably-based pro-$p$ group \cite{NewHorizonsGrigorchuk}. The group $\Gamma_p$ can be defined as the iterated permutational wreath product of the cyclic group of order $p$ and it is a Sylow pro-$p$ subgroup of~$\mathfrak{A}$, where $\mathfrak{A}$ denotes the automorphism group of the $p$-adic tree. There is a strong connection between just infinite branch pro-$p$ groups and Galois groups. Indeed, the Fontaine-Mazur conjecture in number theory \cite{FontaineMazur} implies that the just infinite pro-$p$ quotients of an absolute Galois group of $p$-extensions unramified at~$p$ cannot be $p$-adic analytic \cite{NewHorizonsBoston}. An extension of the Fontaine-Mazur conjecture by Boston proposes that these just infinite pro-$p$ quotients are precisely branch pro-$p$ groups \cite{NewHorizonsBoston}. 

The action of absolute Galois groups on $p$-adic vector spaces has answered long-standing open problems in mathematics, including Fermat's Last Theorem \cite{Wiles}. However the Fontaine-Mazur conjecture suggests that $p$-adic Galois representations are trivial for $p$-extensions unramified at $p$ and therefore new actions of these Galois groups should be investigated. A promising alternative seems to be considering actions on $p$-adic trees \cite{BostonJones}. In fact, these representations of Galois groups on the $p$-adic tree are conjectured to have image of positive Hausdorff dimension in $\Gamma_p$, as an analogue of the finite-index image for $p$-adic representations. This together with the extension of the Fontaine-Mazur conjecture by Boston led to the following purely group-theoretical conjecture of Boston in the spirit of Barnea and Shalev's characterization of $p$-adic analytic groups. Note that from now on a $p$-adic representation of a group $G$ will refer to an embedding of $G$ in $\Gamma_p$.

\begin{Conjecture}[see {\cite[Problem 9]{NewHorizonsBoston}}]
\label{Conjecture: New horizons}
A just infinite pro-$p$ group is branch if and only if it admits a $p$-adic representation of positive Hausdorff dimension in $\Gamma_p$.
\end{Conjecture}

\cref{Conjecture: New horizons} has remained open for over two decades, with only a small progress  for 1-dimensional subgroups by Abért and Virág \cite[Section 9]{AbertVirag}. Since hereditarily just infinite pro-$p$ groups have trivial Hausdorff dimension in $\Gamma_p$ \cite[Section 15]{NewHorizonsGrigorchuk}, \cref{Conjecture: New horizons} reduces to whether all just infinite branch pro-$p$ groups have positive Hausdorff dimension in $\Gamma_p$ or not. One of the main goals of this paper is to construct just infinite branch pro-$p$ groups with zero Hausdorff dimension in $\Gamma_p$. Our construction is based on a new family of groups introduced also in this paper in order to answer an open problem of Grigorchuk on the Hausdorff dimension of self-similar groups.

The conjecture of Barnea and Shalev suggests that the Hausdorff spectrum may be used to characterize certain pro-$p$ groups. Unfortunately, the Hausdorff spectrum of branch pro-$p$ groups is the full interval $[0,1]$ as proved by Klopsch and Röver \cite[Chapter~8]{KlopschPhD} and there are actually several non-branch groups with complete Hausdorff spectrum too; compare \cite{BarneaMatteo, IkerBenjamin, IkerAnitha1, OhianaAlejandraBenjamin, AndoniJon, KlopschPhD}. In the context of $p$-adic analytic pro-$p$ groups, closed subgroups are again $p$-adic analytic; thus, a natural approach is to restrict the Hausdorff spectrum to a family of subgroups sharing a common property $\mathcal{P}$ with the whole group. Then we define the \textit{$\mathcal{P}$ Hausdorff spectrum} of a profinite group as the set of values of the Hausdorff dimension of its closed subgroups that satisfy property $\mathcal{P}$. A natural choice for the property $\mathcal{P}$ is self-similarity, described in terms of the action of the group on a regular rooted tree; see \cref{section: Preliminaries} for precise definitions and the unexplained terms in the introduction. All known examples of closed self-similar groups of $p$-adic automorphisms have rational Hausdorff dimension \cite{BartholdiHausdorff, pBasilica,JorgeMikel,GGSGustavo, NewHorizonsGrigorchuk,Second,GeneralizedBasilica,SunicHausdorff}, which motivated Grigorchuk to pose the following question. Recall that $\mathfrak{A}$ denotes the automorphism group of a regular rooted tree.

\begin{Question}[see {\cite[Problem 7.1]{GrigorchukFinite}}]
\label{question: Grigorchuk problem 7.1}
What is the self-similar Hausdorff spectrum of the group $\mathfrak{A}$? In particular, is its self-similar Hausdorff spectrum contained in the rationals?
\end{Question}

Let us fix $m\ge 2$ and $\mathfrak{A}$ the automorphism group of the $m$-adic tree. For a subgroup $H\le \mathrm{Sym}(m)$ let the group $W_H$ denote the \textit{iterated permutational wreath product} of $H$, where $\mathrm{Sym}(m)$ is the group of permutations of a set of $m$ elements. The group $W_H$ can be seen as a closed subgroup of $\mathfrak{A}$ and we have the equality $\mathfrak{A}=W_{\mathrm{Sym}(m)}$. For $m=p$ a prime, we get $\Gamma_p:=W_{\langle \sigma\rangle}$ where $\sigma:=(1\,2\,\dotsb \, p)\in \mathrm{Sym}(p)$ and, more generally, we define the group of \textit{$q$-adic automorphisms} $\Gamma_q$ as $\Gamma_q:=W_{\langle \sigma \rangle}$ where $\sigma:=(1\,2\,\dotsb \, q)\in\mathrm{Sym}(q)$ for $q$ a $p$-power.

As soon as one leaves the binary rooted tree it is clear that \cref{question: Grigorchuk problem 7.1} needs to be restricted to subgroups of $p$-adic automorphisms (or $q$-adic automorphisms), as otherwise there are clear counterexamples, like the closure of the Hanoi towers group. This was discussed by \v{S}uni\'{c} in \cite{SunicHausdorff}, where he asked explicitly about the rationality of the Hausdorff dimension of self-similar groups in $\Gamma_p$. 

\cref{question: Grigorchuk problem 7.1} can be further generalized by replacing the property of being self-similar with stronger properties such as being weakly regular branch or regular branch. The only known results on this question are for the case of regular branch groups \cite{BartholdiHausdorff, PenlandMaximalHausdorff}. It is remarkable that almost nothing is known about the Hausdorff dimension of general self-similar groups after almost two decades of Grigorchuk posing \cref{question: Grigorchuk problem 7.1}. This may be due to the lack of ``nice" tools to approach this problem. A main goal of this paper is to provide these tools to finally close this gap and completely determine the Hausdorff spectra of the group of $q$-adic automorphisms, restricted to all the most important subclasses of self-similar subgroups. We obtain the following striking result.

\begin{Theorem}
\label{Theorem: restricted Hausdorff spectra p-adic automorphisms}
Let $\Gamma_{q}$ be the group of $q$-adic automorphisms for $q$ a $p$-power. Then the group $\Gamma_{q}$ has the following four restricted Hausdorff spectra:
\begin{enumerate}[\normalfont(i)]
    \item self-similar, super strongly fractal and level-transitive  Hausdorff spectrum:~$[0,1]$;
    \item weakly regular branch Hausdorff spectrum: $(0,1]$;
    \item regular branch Hausdorff spectrum: $\mathbb{Z}[1/p]\cap(0,1]$;
    \item self-similar branch Hausdorff spectrum: $[0,1]$.
\end{enumerate}
\end{Theorem}

It is worth mentioning that the spectra in (ii) and (iii) remain the same under the extra assumption of super strong fractality. However the self-similar branch groups in (iv) are far from satisfying any fractality condition; see \cref{theorem: characterization of profinite regular branch groups}. Both (ii) and (iv) show the classes of weakly regular branch and self-similar branch groups are much richer and different from the one of regular branch groups than anticipated. For instance, we obtain the first examples in the literature of self-similar branch profinite groups acting on regular rooted trees which are not regular branch.

The proof of \cref{Theorem: restricted Hausdorff spectra p-adic automorphisms} is organized into different steps. First, we derive a universal formula for the computation of the Hausdorff dimension of closed self-similar groups. For an arbitrary subgroup $G\le\mathfrak{A}$ we define the sequence $\{r_n\}_{n\ge 1}$ as
$$r_n:=m\log|G_{n-1}|-\log|G_{n}|+\log|G_1|,$$
and its forward gradient
$$\{s_n\}_{n\ge 1}:=\nabla\{r_n\}_{n\ge 1}=\{r_{n+1}-r_n\}_{n\ge 1}.$$
We note that the sequence $\{s_n\}_{n\ge 1}$ coincides with the so-called \textit{series of obstructions} introduced by Petschick and Rajeev in \cite{GeneralizedBasilica}. The key step is to study the sequences $\{r_n\}_{n\ge 1}$ and $\{s_n\}_{n\ge 1}$ via their ordinary generating functions
$$R_G(x):=\sum_{n=1}^\infty r_nx^n\quad\text{and}\quad S_G(x):=\sum_{n=1}^\infty s_nx^n.$$ 

We say a profinite group has \textit{strong Hausdorff dimension} if its Hausdorff dimension is given by a proper limit. Our main result here is the following.

\begin{Theorem}
\label{theorem: Formula Hausdorff dimension}
Let $G\le W_H$ where $H\le \mathrm{Sym}(m)$. Then, if all the terms of the sequence $\{r_n\}_{n\ge 1}$ are of equal sign, the closure of $G$ has strong Hausdorff dimension in $W_H$. What is more, we have
$$\mathrm{hdim}_{W_H}(\overline{G})=\frac{1}{\log |H|}\big(\log|G_1|-S_G(1/m)\big).$$
\end{Theorem}

 \cref{theorem: Formula Hausdorff dimension} would be of little value if the groups satisfying its hypothesis on the sequence $\{r_n\}_{n\ge 1}$ were rare in the literature. Luckily, this is not the case, as in particular this condition is satisfied by both self-similar and branching groups, where branching groups are those containing copies of themselves and play an important role in the proof of \cref{Theorem: restricted Hausdorff spectra p-adic automorphisms} too; see \cref{proposition: weakly branch Hausdorff}.

\begin{proposition}
\label{proposition: self-similar and branching formula}
    Let $G\le\mathfrak{A}$ be weakly self-similar (or branching). Then the sequence $\{r_n\}_{n\ge 1}$ is non-negative and increasing (respectively non-positive and decreasing) and $G$ satisfies the assumptions in \cref{theorem: Formula Hausdorff dimension}.
\end{proposition}

We remark that these tools are interesting in their own right. To illustrate this, we combine \cref{theorem: Formula Hausdorff dimension} and \cref{proposition: self-similar and branching formula} to answer a question of Bartholdi on the finitely generated self-similar Hausdorff spectrum of $\mathfrak{A}$; see \cite[Question 1]{BartholdiHausdorff}.

\begin{theorem}
\label{theorem: fg spectrum}
For a non-perfect subgroup $H\le \mathrm{Sym}(m)$ the group $W_H$ does not contain a finitely generated self-similar subgroup of Hausdorff dimension 1. In particular, the group $\mathfrak{A}$ does not contain a finitely generated self-similar subgroup of Hausdorff dimension 1.
\end{theorem}

Next, we generalize iterated permutational wreath products as follows. We choose a subgroup $S_0\le \mathrm{Sym}(q)$ and we set $A_0:=S_0$. For every $n\ge 1$, instead of choosing the base group $S_0\times\overset{q^n}{\ldots}\times S_0$ and constructing its semidirect product with $A_{n-1}$ as we do for $W_{S_0}$, we consider an $A_{n-1}$-invariant subgroup $S_n$ and define $A_n:=S_n\rtimes A_{n-1}$. Then $G_\mathcal{S}:=\dotsb \rtimes S_1\rtimes S_0$, where $\mathcal{S}:=\{S_n\}_{n\ge 0}$ is the \textit{defining sequence} of the group $G_\mathcal{S}$.

We study the properties of the group $G_\mathcal{S}$ in terms of the defining sequence~$\mathcal{S}$ in \cref{section: generalized iterated wreath products}. In \cref{section: restricted Hausdorff spectra}, we show how the sequence $\mathcal{S}$ can be chosen so that $G_\mathcal{S}\le \Gamma_q$ and $S_{n}$ has index $q^{\gamma_n}$ in $S_{n-1}\times\overset{q}{\ldots}\times S_{n-1}$ for any prescribed sequence $\{\gamma_n\}_{n\ge 1}$ where $0\le \gamma_n\le q-1$ for all $n\ge 1$. By \cref{theorem: Formula Hausdorff dimension}, the generating function $S_{G_\mathcal{S}}(x)$ at $x=1/q$ gives precisely the $q$-base expansion of $1-\mathrm{hdim}_{\Gamma_q}(G_\mathcal{S})$ and we obtain a method to construct closed self-similar groups with prescribed Hausdorff dimension $\mathrm{hdim}_{\Gamma_q}(G_\mathcal{S})\in [0,1]$. Adding additional restrictions to the sequence $\{\gamma_n\}_{n\ge 1}$ we obtain (ii) and (iii) in \cref{Theorem: restricted Hausdorff spectra p-adic automorphisms}. For (iv) we use a slightly more involved construction for the defining sequence $\mathcal{S}$ but the underlying idea remains the same.

A natural question to ask at this point is how the above restricted Hausdorff spectra of $q$-adic automorphisms look like under the extra assumption of the closed subgroups being also topologically finitely generated.

\begin{Question}
\label{question: finitely generated spectra}
    What are the restricted Hausdorff spectra of the group of $q$-adic automorphisms in \cref{Theorem: restricted Hausdorff spectra p-adic automorphisms} if we add the topologically finitely generated condition?
\end{Question}

To the best of the author's knowledge, \cref{theorem: fg spectrum} is the only known general result on the finitely generated self-similar Hausdorff spectrum of the group $\Gamma_p$. Abért and Virág proved that three random elements of~$\Gamma_p$ generate topologically a closed subgroup of Hausdorff dimension 1 in $\Gamma_p$ with probability~1 \cite[Theorem~7.2]{AbertVirag}. \cref{theorem: fg spectrum} shows that for closed self-similar subgroups of $p$-adic automorphisms we get the completely opposite picture as there is not a single finitely generated closed self-similar subgroup of Hausdorff dimension 1 in $\Gamma_p$. Thus we expect quite a different result for the finitely generated case. The pro-$p$ groups $G_\mathcal{S}$ are not appropriate to answer \cref{question: finitely generated spectra} as they are never topologically finitely generated; see \cref{proposition: not finitely generated}. However, we believe \cref{theorem: Formula Hausdorff dimension} could potentially be useful in answering \cref{question: finitely generated spectra}.

The rich interplay between group theory, formal languages and complex and symbolic dynamics has been studied throughout the last decades; see \cite{Holom, MarialauraBartholdi, GrigorchukFinite, PenlandSunic1, PenlandSunic2, SunicHausdorff}. The tools developed in the proof of \cref{theorem: Formula Hausdorff dimension} provide further applications of self-similar groups to measure-preserving dynamical systems. In fact, in an ongoing work of the author, the sequence $\{r_n\}_{n\ge 1}$ plays a key role in generalizing Bowen's classification of Bernoulli shifts over free groups in \cite{BowenF} to Markov processes over free semigroups \cite{BowenMarkov} arising from fractal regular branch profinite groups. Further applications to arithmetic dynamics have also been obtained recently in an ongoing project of the author together with Radi.

Coming back to \cref{Conjecture: New horizons}, the zero-dimensional closed branch groups $G_\mathcal{S}$ suggest a way to disprove \cref{Conjecture: New horizons}. If we encode the action of all $S_n\in \mathcal{S}$ in a single automorphism we can obtain finite (topological) generation and we may potentially construct just infinite branch pro-$p$ groups with Hausdorff dimension zero in $\Gamma_p$. We show this can be done in a similar way to how the multi GGS-groups are constructed \cite{FG}, which yields for every prime $p$ a family of discrete groups acting on the $p$-adic tree such that each group is finitely generated, branch, just infinite and of zero-dimensional closure:

\begin{Theorem}
    \label{Theorem: just infinite trivial dimension}
    For any prime $p$ there exists a just infinite branch pro-$p$ group of $p$-adic automorphisms with zero Hausdorff dimension in $\Gamma_p$. 
\end{Theorem}

 The construction in \cref{Theorem: just infinite trivial dimension} does not disprove yet \cref{Conjecture: New horizons} because one still needs to show that all the other branch actions of the groups yield zero-dimensional $p$-adic representations in $\Gamma_p$. Although we cannot prove this, it is our hope that the construction in \cref{Theorem: just infinite trivial dimension} does provide a counterexample to \cref{Conjecture: New horizons}.

\subsection*{\textit{\textmd{Organization}}} 
In \cref{section: Preliminaries} we give some background for the subsequent sections. In \cref{section: new tool for the Hausdorff dimension} we prove \cref{theorem: Formula Hausdorff dimension} and show some applications. Then in \cref{section: generalized iterated wreath products} we introduce a generalization of iterated wreath products and study their properties and in \cref{section: restricted Hausdorff spectra} we prove \cref{Theorem: restricted Hausdorff spectra p-adic automorphisms}. Later in \cref{section: branch groups} we construct just infinite branch groups whose closure has zero Hausdorff dimension in $\Gamma_p$ proving \cref{Theorem: just infinite trivial dimension}.

\subsection*{\textit{\textmd{Notation}}} All the logarithms will be taken in base $m$, where $m$ is the degree of the regular rooted tree. The number $p$ will denote a prime integer and $q$ a $p$-power. For a sequence $\{a_i\}_{i\in I}$ we denote by $\nabla \{a_i\}:=\{a_{i+1}-a_i\}_{i\in I}$ its forward gradient. We denote by $\sigma:=(1\,\dotsb \, m)\in \mathrm{Sym}(m)$ the $m$-cycle acting on the set $\{1,\dotsc,m\}$. Maps will be considered to act on the right and we write composition from left to right, i.e. $fg$ means first apply $f$ and then apply $g$. Lastly, the topology of any group acting on the $m$-adic tree is always assumed to be the \textit{congruence topology}, i.e. the one given by its level stabilizers.

\subsection*{Acknowledgements} I would like to thank my advisors Gustavo A. Fernández-Alcober and Anitha Thillaisundaram for their valuable feedback and helpful discussions which contributed greatly to the final version of this paper. I would also like to thank J.~Moritz Petschick for helpful discussions and for pointing out a gap in a previous version of this paper. Finally, let me thank the anonymous referee for their valuable feedback and suggestions.


\section{Background}
\label{section: Preliminaries}

\subsection{The $m$-adic tree and its group of automorphisms} We define the $m$-\textit{adic tree} as the infinite regular rooted tree where each vertex has exactly $m$ descendants, where $m\ge 2$. Each vertex can be identified with a finite word on letters in a set $X$ of $m$ elements. The set $X$ may be completely ordered, inducing a lexicographical order in the $m$-adic tree. Two vertices $u$ and~$v$ are joined by an edge if $u=vx$ with $x\in X$. In this case we say $u$ is \textit{below} $v$. The \textit{$n$th level of the tree} is defined to be the set of all vertices which consist of words of length $n$. We may also use the word level to refer to the number $n$. Let $\mathfrak{A}$ be the group of graph automorphisms of the $m$-adic tree, i.e. those bijective maps on the set of vertices of the tree preserving adjacency. Since we are considering regular rooted trees such automorphisms fix the root and permute the vertices at the same level of the tree. For any natural number $k\ge 1$, the \textit{$k$th truncated tree} consists of the vertices at distance at most $k$ from the root. We denote the group of automorphisms of the $k$th truncated tree by $\mathfrak{A}_{[k]}$.

Let $g\in\mathfrak{A}$. For any vertex $v$ and any integer $1\le k\le \infty$ we define the \textit{section of $g$ at $v$ of depth }$k$ as the unique automorphism $g|_v^k\in\mathfrak{A}_{[k]}$ such that $(vu)g=(v)g(u)g|_v^k$ for every $u$ of length at most $k$. For $k=\infty$ we write  $g|_v$ for simplicity and call it the \textit{section of $g$ at $v$}, while for $k=1$ we shall call $g|_v^1$ the \textit{label of $g$ at $v$}.

For every $n\ge 1$ the subset $\mathrm{St}(n)$ of automorphisms fixing all the vertices of the $n$th level of the tree forms a normal subgroup of finite index in $\mathfrak{A}$, the \textit{$n$th level stabilizer}. Similarly, for any vertex $v$ of the tree we define its \textit{vertex stabilizer} $\mathrm{st}(v)$ as the subgroup of automorphisms fixing the vertex $v$. The group $\mathfrak{A}$ is residually finite as $\bigcap \mathrm{St}(n)=1$ and a countably-based profinite group with respect to the topology induced by this filtration of the level stabilizers. For each vertex $v$ we shall define the continuous homomorphism $\psi_v:\mathrm{st}(v)\to \mathfrak{A}$ given by $g\mapsto g|_v$ and for each integer $n\ge 1$ we define the continuous isomorphisms $\psi_n:\mathrm{St}(n)\to \mathfrak{A}\times\overset{m^n}{\ldots}\times \mathfrak{A}$ given by $g\mapsto (g|_{v_1},\dotsc,g|_{v_{m^n}})$, where $v_1,\dotsc,v_{m^n}$ denote the vertices at the $n$th level of the tree from left to right according to the lexicographical order. We shall write $\psi:=\psi_1$ for simplicity.

Observe that $\mathfrak{A}=\mathrm{St}(n)\rtimes \mathfrak{A}_n$, where $\mathfrak{A}_n$ is the subgroup of \textit{finitary automorphisms of depth $n$}, i.e. automorphisms of the $m$-adic tree with trivial labels at the $k$th level of the tree for every $k\ge n$. This observation allows one to extend the isomorphisms $\psi_n:\mathrm{St}(n)\to \mathfrak{A}\times\overset{m^n}{\ldots}\times \mathfrak{A}$ to isomorphisms $\psi_n:\mathfrak{A}\to (\mathfrak{A}\times\overset{m^n}{\ldots}\times \mathfrak{A})\rtimes \mathfrak{A}_{n}$ which we shall also denote by $\psi_n$ by a slight abuse of notation. Finitary automorphisms of depth 1 are called \textit{rooted automorphisms} and they may be identified with elements in the symmetric group $\mathrm{Sym}(m)$. More generally, the group $\mathfrak{A}_n$ is isomorphic to the group of automorphisms of the $n$th truncated tree $\mathfrak{A}_{[n]}$ via the restriction of the action of $\mathfrak{A}$ to the $n$th truncated tree.

\subsection{Subgroups of $\mathfrak{A}$}
For a subgroup $G\le \mathfrak{A}$ we define $\mathrm{st}_G(v):=\mathrm{st}(v)\cap G$ and $\mathrm{St}_G(n):=\mathrm{St}(n)\cap G$ for any vertex $v$ and level $n\ge 1$ respectively. The quotients $G_n:=G/\mathrm{St}_G(n)$ are the \textit{congruence quotients} of $G$.

Let $G$ be a subgroup of $\mathfrak{A}$. We say that $G$ is \textit{self-similar} if for any $g\in G$ we have $g|_{v}\in G$ for any vertex $v$ and \textit{weakly self-similar} if $g|_v\in G$ for $g\in \mathrm{St}_G(1)$ and any vertex $v$. We shall say that $G$ is \textit{fractal} (or \textit{strongly fractal}) if $G$ is self-similar and $\psi_v(\mathrm{st}_G(v))=G$ (respectively $\psi_v(\mathrm{St}_G(1))=G$) for any $v$ at the first level of the tree. An even stronger version of fractality is that of \textit{super strongly fractal} groups, where $\psi_v(\mathrm{St}_G(n))=G$ for every vertex $v$ at the $n$th level of the tree for every level $n\ge 1$. We also introduce a weaker version of fractality too. We say that $G$ is \textit{`virtually' fractal} if $G$ is weakly self-similar and  $\psi_v(\mathrm{St}_G(1))$ is of finite index in $G$ for any $v$ at the first level of the tree. Note that we write `virtually' because our definition differs slightly from the standard notion of virtually $\mathcal{P}$ for some property~$\mathcal{P}$. We have the following chain of implications
$$\text{super strongly fractal}\implies \text{strongly fractal}\implies \text{fractal}\implies \text{`virtually' fractal}.$$

We say that $G$ is \textit{level-transitive} if $G$ acts transitively on every level of the tree. For any vertex $v$ let $\mathrm{rist}_G(v)$ be the associated \textit{rigid vertex stabilizer}, i.e. the subgroup of $\mathrm{st}_G(v)$ consisting of automorphisms $g\in \mathrm{st}_G(v)$ such that for any vertex $w$ which is not below $v$ we have $g|_w^1=1$. Then for any $n\ge 1$ we define the corresponding \textit{rigid level stabilizers} $\mathrm{Rist}_G(n)$ as the product of all rigid vertex stabilizers of the vertices at the level~$n$. Rigid level stabilizers are normal subgroups of $G$ as any conjugate of a rigid vertex stabilizer is a rigid vertex stabilizer of the same level. We say that $G$ is \textit{weakly branch} if $G$ is level-transitive and its rigid level stabilizers are non-trivial for every level, and \textit{branch} if they are furthermore of finite index in $G$. The stronger notions of \textit{weakly regular branch} and \textit{regular branch} are defined as $G$ being a self-similar, level-transitive group containing a non-trivial (respectively finite-index) branching subgroup $K$, where a \textit{branching group} $K$ is a subgroup of $\mathfrak{A}$ such that $\psi(\mathrm{St}_K(1))\ge K\times\dotsb\times K$.

\subsection{Just infinite groups}
A \textit{just infinite discrete} group is an infinite discrete group such that all its proper quotients are finite. We define \textit{just infinite profinite} groups as those infinite profinite groups whose normal non-trivial closed subgroups are open, i.e. infinite profinite groups whose continuous proper quotients are all finite.

\subsection{Hausdorff dimension}
Countably-based profinite groups are metrizable which allows the definition of Hausdorff measures and therefore a Hausdorff dimension for its subsets. It was proved by Barnea and Shalev in \cite{BarneaShalev} based on the work of Abercrombie \cite{Abercrombie} that for a profinite group the Hausdorff dimension of its closed subgroups coincides with their lower box dimension. For the group $\mathfrak{A}$ the most natural choice for the metric is the one induced by the level stabilizers and in this context the formula for the Hausdorff dimension of a closed subgroup $G\le \mathfrak{A}$ reads as
$$\mathrm{hdim}_\mathfrak{A}(G)=\liminf_{n\to \infty}\frac{\log|G_n|}{\log|\mathfrak{A}_n|}.$$
For a subgroup $H\le \mathrm{Sym}(m)$, we define the \textit{iterated permutational wreath product} $W_H\le \mathfrak{A}$ as $W_H:=\{g\in \mathfrak{A}\mid g|_v^1\in H \text{ for every }v\}$. Note that an alternative construction of $W_H$ is $W_H=\varprojlim H\wr \overset{n}{\dotsb}\wr H$ and that the groups $\mathfrak{A}$ and $\Gamma_q$ correspond respectively to $W_{\mathrm{Sym}(m)}$ and $W_{\langle \sigma\rangle}$ where $\sigma:=(1\,2\,\dotsb \, q)\in \mathrm{Sym}(q)$. 

For a group $G\le W_H$ the Hausdorff dimension of $G$ may also be computed relative to $W_H$ as
\begin{align}
\label{align: hausdorff dimension in WH}
    \mathrm{hdim}_{W_H}(G)&=\liminf_{n\to \infty}\frac{\log|G_n|}{\log|(W_H)_n|}=\frac{1}{\mathrm{hdim}_\mathfrak{A}(W_H)}\mathrm{hdim}_\mathfrak{A}(G)\\
    &=\frac{\log m!}{\log |H|}\mathrm{hdim}_\mathfrak{A}(G).\nonumber
\end{align}

Let $\mathcal{P}$ be a property that may be satisfied by the closed subgroups of a profinite group. We define the \textit{$\mathcal{P}$ Hausdorff spectrum} of a profinite group as the set of the Hausdorff dimensions of its closed subgroups that satisfy property $\mathcal{P}$. We shall study the $\mathcal{P}$ Hausdorff spectra of the group $\Gamma_q$ where $\mathcal{P}$ is one of the different properties appearing in the statement of \cref{Theorem: restricted Hausdorff spectra p-adic automorphisms}.

\section{Hausdorff dimension formula and applications}
\label{section: new tool for the Hausdorff dimension}

Petschick and Rajeev introduced a series of obstructions in \cite{GeneralizedBasilica} in order to compute the Hausdorff dimension of groups obtained via the so-called Basilica operation. Their definition of this sequence is rather technical and it seems useful due to its nice behaviour with respect to the Basilica operation \cite[Proposition 4.16]{GeneralizedBasilica}. However, this sequence arises naturally in the study of the Hausdorff dimension of weakly self-similar groups. 

\subsection{Definition of the sequences}
Let $G\le \mathfrak{A}$. Recall the sequence $\{r_n(G)\}_{n\ge 1}$ is defined as
$$r_n(G):=m\log|G_{n-1}|-\log|G_{n}|+\log|G_1|.$$
Note that $r_1(G)=0$ for any $G\le \mathfrak{A}$. The sequence $\{s_n(G)\}_{n\ge 1}$ is defined as
$$s_n(G):=r_{n+1}(G)-r_n(G)=m\log|\mathrm{St}_G(n-1):\mathrm{St}_G(n)|-\log|\mathrm{St}_G(n):\mathrm{St}_G(n+1)|.$$
Note $\{s_n(G)\}_{n\ge 1}$ coincides with the sequence introduced by Petschick and Rajeev in \cite{GeneralizedBasilica}. In general we omit the reference to the group $G$ when there is no room for confusion and we simply write $r_n$ and $s_n$ for the terms of these sequences. Note that both sequences $\{r_n\}_{n\ge 1}$ and $\{s_n\}_{n\ge 1}$ are invariant under taking topological closures as they only depend on the congruence quotients and
$$\frac{\overline{G}}{\mathrm{St}_{\overline{G}}(n)}\cong \frac{G}{\mathrm{St}_G(n)}$$
for every $n\ge 1$ by a standard result on profinite groups \cite[Proposition 3.2.2(d)]{RibesZalesskii}. Therefore these sequences do not distinguish between a group and its closure.

 We shall see that the key is to study the sequences $\{r_n\}_{n\ge 1}$ and $\{s_n\}_{n\ge 1}$ via their ordinary generating functions
$$R_G(x):=\sum_{n=1}^\infty r_nx^n\quad\text{and}\quad S_G(x):=\sum_{n=1}^\infty s_nx^n.$$ 

\subsection{Motivation and main cases}
\label{subsection: motivation and main cases}

Now recall \v{S}uni\'{c}'s formula \cite[Theorem 4]{SunicHausdorff} to compute the Hausdorff dimension of a closed regular branch group $G$ branching over its $k$th level stabilizer:
\begin{align*}
    \mathrm{hdim}_{W_H}(G)=\frac{(m-1)\log|G_k|-\log|G\times\dotsb\times G:\psi(\mathrm{St}_G(1))|+\log|G_1|}{m^k}.
\end{align*}
Therefore, in order to compute the Hausdorff dimension of a closed regular branch group~$G$ it is enough to understand what happens above $\psi(\mathrm{St}_G(1))$. Then we define the subgroups $$\mathfrak{S}_n:=\psi(\mathrm{St}_G(1))\big(\mathrm{St}_G(n-1)\times\overset{m}{\dotsb}\times\mathrm{St}_G(n-1)\big).$$
These are natural subgroups to consider as for a closed weakly self-similar group $G$ the image $\psi(\mathrm{St}_G(1))$ is closed in $G\times\overset{m}{\ldots}\times G$ so the subgroups $\mathfrak{S}_n$ give a natural filtration 
$$G\times\overset{m}{\dotsb}\times G= \mathfrak{S}_1\ge\dotsb\ge \mathfrak{S}_n\ge \dotsb\ge \bigcap_{n\ge 1} \mathfrak{S}_n=\overline{\psi(\mathrm{St}_G(1))}=\psi(\mathrm{St}_G(1)).$$

The sequences $\{r_n\}_{n\ge 1}$ and $\{s_n\}_{n\ge 1}$ have a very natural definition in terms of the subgroups $\mathfrak{S}_n$ simply as
\begin{align}
\label{align: rn and sn self-similar}
    r_n=\log |G\times\overset{m}{\dotsb}\times G: \mathfrak{S}_{n}|\quad\text{and}\quad s_n=\log|\mathfrak{S}_{n}:\mathfrak{S}_{n+1}|.
\end{align}
This equivalent definition for weakly self-similar groups follows directly from the equality
$$\psi(\mathrm{St}_G(n))=\psi(\mathrm{St}_G(1))\cap\big(\mathrm{St}_G(n-1)\times\overset{m}{\dotsb}\times \mathrm{St}_G(n-1)\big)$$
and the Second Isomorphism Theorem, see \cref{fig:my_label} below.

\begin{figure}[H]
    \centering
    \begin{tikzpicture}
    \node (stn) at (0,0) {$\mathrm{St}_G(n)$};
    \node (st1) at (0,1) {$\mathrm{St}_G(1)$};
    \node (G) at (0,2) {$G$};
    \node (psi1) at (3,1) {$\psi(\mathrm{St}_G(1))$};
    \node (psin) at (5,0) {$\psi(\mathrm{St}_G(n))$};
    \node (dirn-1) at (8,1) {$\mathrm{St}_G(n-1)\times\overset{m}{\dotsb}\times\mathrm{St}_G(n-1)$};
    \node (Sn) at (5,2) {$\mathfrak{S}_n$};
    \node (dirG) at (5,3) {$G\times\overset{m}{\dotsb}\times G$};
    \draw [-] (G) edge node[swap] {} (st1);
    \draw [-] (st1) edge node[swap] {} (stn);
    \draw [-] (dirG) edge node[swap] {} (Sn);
    \draw [-] (Sn) edge node[swap] {} (psi1);
    \draw [-] (Sn) edge node[swap] {} (dirn-1);
    \draw [-] (psi1) edge node[swap] {} (psin);
    \draw [-] (dirn-1) edge node[swap] {} (psin);
    \draw [->] (st1) edge node[swap] {} (psi1);
    \draw [->] (stn) edge node[swap] {} (psin);
    \draw[<->, color=red] (dirG.east) to[out=0,in=90,looseness=1] (dirn-1.north);
    \draw[<->, color=teal] (st1.west) to[out=180,in=180,looseness=1] (stn.west);
    \draw[<->, color=teal] (psi1.east) to[out=0,in=90,looseness=1] (psin.north);
    \draw[<->, color=teal] (Sn.east) to[out=0,in=140,looseness=1] (dirn-1.110);
    \draw[<->, color=violet] (dirG.200) to[out=220,in=180,looseness=1] (Sn.west);
    \node (index1) at (-1.4,0.5) {\textcolor{teal}{$\frac{|G_n|}{|G_1|}$}};
    \node (index2) at (8.3,2.7) {\textcolor{red}{$|G_{n-1}|^m$}};
    \node (rn) at (3.9,2.1) {\textcolor{violet}{$m^{r_n}$}};
    \node (psi1label) at (1.4,1.2) {$\psi$};
    \node (psi1label2) at (1.4,0.2) {$\psi$};
    \end{tikzpicture}
    \caption{The definition of $r_n$ for a weakly self-similar group in terms of the subgroup $\mathfrak{S}_n$, where the same colours represent the same indices.}
    \label{fig:my_label}
\end{figure}
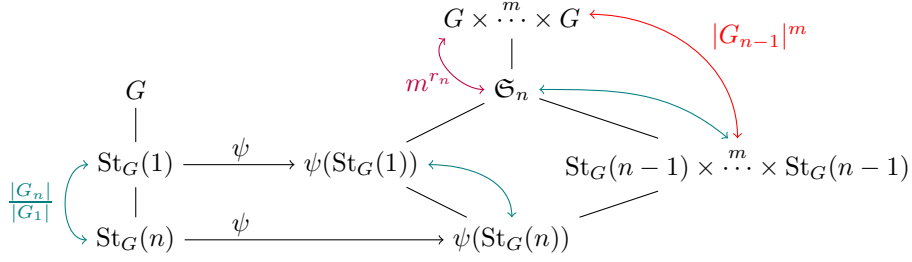

Similarly, for a branching group $K\le \mathfrak{A}$ we obtain an equivalent definition for both sequences $\{r_n\}_{n\ge 1}$ and $\{s_n\}_{n\ge 1}$ which resembles the one above for weakly self-similar groups. We define for every $n\ge 1$ the subgroup
$$\mathfrak{B}_n:=\psi(\mathrm{St}_K(n))\big(K\times\overset{m}{\dotsb}\times K\big)$$
and then we have
\begin{align}
\label{align: rn and sn branching}
    r_n=-\log |\mathrm{St}_K(1) : \mathfrak{B}_{n}|\quad\text{and}\quad s_n=-\log|\mathfrak{B}_{n}:\mathfrak{B}_{n+1}|.
\end{align}
Again these alternative definitions of the sequences $\{r_n\}_{n\ge 1}$ and $\{s_n\}_{n\ge 1}$ follow from the equality
$$\mathrm{St}_K(n-1)\times\overset{m}{\dotsb}\times \mathrm{St}_K(n-1)=\psi(\mathrm{St}_K(n))\cap \big(K\times\overset{m}{\dotsb}\times K\big)$$
and the Second Isomorphism Theorem, see \cref{fig: branching} below.

\begin{figure}[H]
    \centering
    \begin{tikzpicture}
    \node (stn) at (0,0) {$\mathrm{St}_K(n)$};
    \node (st1) at (0,2) {$\mathrm{St}_K(1)$};
    \node (G) at (0,3) {$K$};
    \node (psi1) at (3,2) {$\psi(\mathrm{St}_K(1))$};
    \node (psin) at (3,0) {$\psi(\mathrm{St}_K(n))$};
    \node (dirn-1) at (8,0) {$K\times\overset{m}{\dotsb}\times K$};
    \node (Sn) at (5,1) {$\mathfrak{B}_n$};
    \node (dirn) at (5,-1) {$\mathrm{St}_K(n-1)\times\overset{m}{\dotsb}\times \mathrm{St}_K(n-1)$};
    \draw [-] (G) edge node[swap] {} (st1);
    \draw [-] (st1) edge node[swap] {} (stn);
    \draw [-] (Sn) edge node[swap] {} (psi1);
    \draw [-] (Sn) edge node[swap] {} (psin);
    \draw [-] (Sn) edge node[swap] {} (dirn-1);
    \draw [-] (psin) edge node[swap] {} (dirn);
    \draw [-] (dirn) edge node[swap] {} (dirn-1);
    \draw [->] (st1) edge node[swap] {} (psi1);
    \draw [->] (stn) edge node[swap] {} (psin);
    \draw[<->, color=red] (Sn.310) to[out=-45,in=0,looseness=1] (psin.east);
    \draw[<->, color=red] (dirn-1.330) to[out=-45,in=0,looseness=1] (dirn.east);
    \draw[<->, color=teal] (st1.west) to[out=180,in=180,looseness=1] (stn.west);
    \draw[<->, color=teal] (psi1.200) to[out=200,in=160,looseness=1] (psin.160);
    \draw[<->, color=violet] (psi1.0) to[out=0,in=90,looseness=1] (Sn.north);
    \node (index1) at (-1.7,1) {\textcolor{teal}{$\frac{|K_n|}{|K_1|}$}};
    \node (index2) at (9.1,-1) {\textcolor{red}{$|K_{n-1}|^m$}};
    \node (rn) at (5.2,2) {\textcolor{violet}{$m^{-r_n}$}};
    \node (psi1label) at (1.4,2.2) {$\psi$};
    \node (psi1label2) at (1.4,0.2) {$\psi$};
    \end{tikzpicture}
    \caption{The definition of $r_n$ for a branching group in terms of the subgroup $\mathfrak{B}_n$, where the same colours represent the same indices.}
    \label{fig: branching}
\end{figure}
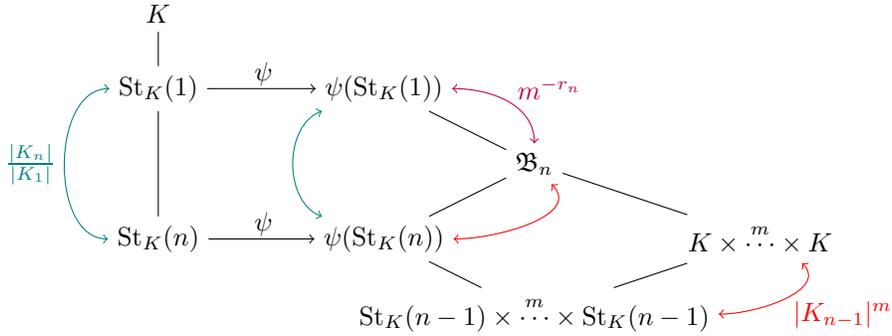

Thus for weakly self-similar groups the sequence $\{r_n\}_{n\ge 1}$ is non-negative and increasing by \cref{align: rn and sn self-similar}, while for branching groups  the sequence $\{r_n\}_{n\ge 1}$ is non-positive and decreasing by \cref{align: rn and sn branching}. This proves \cref{proposition: self-similar and branching formula}.

\subsection{Proof of \cref{theorem: Formula Hausdorff dimension}}

We first prove that the closure of a group $G$ satisfying the assumptions in \cref{theorem: Formula Hausdorff dimension} has strong Hausdorff dimension. We use monotone convergence to establish the existence of this limit. Then we show that this limit may be given in terms of the value of the generating function $S_G(x)$ at $x=1/m$.

\begin{lemma}
\label{lemma: formula in terms of rn of the congruence quotients orders}
    For $G\le\mathfrak{A}$ we have
    $$\log|G_n|=\sum_{i=1}^n(\log|G_1|-r_i)m^{n-i}=\left(\frac{m^n-1}{m-1}\right)\log|G_1|-\sum_{i=1}^n r_i m^{n-i}$$
    for all $n\ge 1$.
\end{lemma}
\begin{proof}
    For $n=1$ the result is clear as $r_1=0$. Since $$\log|G_n|=m\log|G_{n-1}|+(\log|G_1|-r_n),$$
    the first equality follows by induction on $n\ge 1$. The second equality follows directly from the first one by rearranging terms.
\end{proof}

A similar version of the above result appeared in the proof of \cite[Lemma 4.13]{GeneralizedBasilica} for the corresponding series of obstructions.

Now we add the condition that the terms $r_n$ are all of equal sign. This is a natural assumption as self-similar groups satisfy this condition, as we saw at the beginning of this section, and also because groups which are known not to have strong Hausdorff dimension do not fulfill this condition; see the spinal groups in \cite{SiegenthalerHausdorff} for examples.

For every $n\ge 1$, let us set
$$L_n:=\frac{1}{m^n}\sum_{i=1}^nr_im^{n-i}=\sum_{i=1}^{n}\frac{r_i}{m^i}.$$
The next lemma shows the sequence $\{L_n\}_{n\ge 1}$ has a limit  $L$ which allows an immediate calculation of the Hausdorff dimension of the closure of the group $G$.

\begin{lemma}
    \label{lemma: monotonic sequence}
    Assume $G\le\mathfrak{A}$ is such that all the terms of its sequence $\{r_n\}_{n\ge 1}$ are of equal sign meaning they are either all non-negative or all non-positive. Then the limit
    $$L=\lim_{n\to\infty} L_n$$
    exists. Furthermore, the closure of $G$ has strong Hausdorff dimension in $\mathfrak{A}$ given by 
    $$\mathrm{hdim}_{\mathfrak{A}}(\overline{G})=\frac{1}{\log m!}\left(\log|G_1|-(m-1)L\right).$$
\end{lemma}
\begin{proof}
Note that since $G_n$ is a subgroup of $\mathfrak{A}_n$ we have $0\le \log|G_n|/\log|\mathfrak{A}_n|\le 1$ for every $n\ge 1$. Applying the formula in \cref{lemma: formula in terms of rn of the congruence quotients orders} we obtain the inequalities
    \begin{align}
    \label{align: density expression}
    0\le \frac{\log|G_n|}{\log|\mathfrak{A}_n|}&=\frac{\left(\frac{m^n-1}{m-1}\right)\log|G_1|-\sum_{i=1}^n r_i m^{n-i}}{\left(\frac{m^n-1}{m-1}\right)\log m!}\nonumber\\
    &=\frac{1}{\log m!}\left(\log|G_1|-(m-1)\left(\frac{m^n}{m^n-1}\right)L_n\right)\le 1.
\end{align}
Hence the formula for the Hausdorff dimension will be clear once we establish the convergence of the sequence $\{L_n\}_{n\ge 1}$. We have $(m-1)/m\le (m^n-1)/m^n\le 1$ for all $n\ge 1$. Then, applying these bounds and rearranging terms in \textcolor{teal}{(}\ref{align: density expression}\textcolor{teal}{)} we obtain the inequalities
$$\frac{\log|G_1|-\log m!}{m}\le L_n\le \frac{\log|G_1|}{m-1},$$
i.e. the terms $L_n$ are all bounded for $n\ge 1$. Without loss of generality let us assume $r_n\ge 0$ for all $n\ge 1$; otherwise replace $r_n$ with $-r_n$ for all $n\ge 1$ and work with $-L_n$ instead of $L_n$. Then $L_n\ge 0$ and it is increasing as
\begin{align*}
    L_{n+1}&=\sum_{i=1}^{n+1}\frac{r_i}{m^i}=\frac{r_{n+1}}{m^{n+1}}+L_n\ge L_n
\end{align*}
since $r_{n}\ge 0$ for all $n\ge 1$. Thus the sequence $\{L_n\}_{n\ge 1}$ is increasing and bounded and by the Monotone Convergence Theorem it has a limit $L$.
\end{proof}

To conclude the proof of \cref{theorem: Formula Hausdorff dimension} we need to establish that $S_G(1/m)$ is well defined and that it coincides with $(m-1)L$.

\begin{lemma}
\label{lemma: convergence of the Hausdorff series}
    Assume $G\le\mathfrak{A}$ is such that all the terms of the sequence $\{r_n\}_{n\ge 1}$ are of equal sign. Then the power series $S_G(x)$ converges at $x=1/m$ and this value satisfies $S_G(1/m)=(m-1)L$. 
\end{lemma}
\begin{proof}
    Note that the generating function of the sequence $\{r_n\}_{n\ge 1}$ converges at $x=1/m$ by \cref{lemma: monotonic sequence} as
    $$R_G(1/m)=\sum_{n=1}^\infty \frac{r_n}{m^n}=\lim_{n\to\infty}L_n=L.$$
    We claim that we have the equality $xS_G(x)=(1-x)R_G(x)$, and hence evaluating at $x=1/m$ yields the result. The claimed equality follows from 
    \begin{align*}
        xS_G(x)&=\sum_{n\ge 1}(r_{n+1}-r_n)x^{n+1}=R_G(x)-r_1x-xR_G(x)=(1-x)R_G(x),
    \end{align*}
    as $r_1=0$.
\end{proof}

\begin{proof}[Proof of \cref{theorem: Formula Hausdorff dimension}]
\cref{lemma: monotonic sequence} together with \cref{lemma: convergence of the Hausdorff series} yield both the strong Hausdorff dimension statement and the formula in \cref{theorem: Formula Hausdorff dimension} for $\mathfrak{A}$. The general formula for $W_H$ follows directly from the one of $\mathfrak{A}$ and \textcolor{teal}{(}\ref{align: hausdorff dimension in WH}\textcolor{teal}{)}.
\end{proof}

\subsection{A characterization of regular branchness}  A group $G$ is \textit{commensurable} to another group $H$ when there exist finite-index subgroups $K\le G$ and $L\le H$ and an isomorphism $f:K\to L$. The following result delivers some characterizations of closed regular branch groups, and thus answers a question of Siegenthaler \cite[Section 1.8]{SiegenthalerPhD} asking whether a profinite group $G$ is regular branch if and only if $ \psi(\mathrm{St}_G(1))$ has finite index in $G\times \overset{m}{\ldots}\times G$.

\begin{theorem}
\label{theorem: characterization of profinite regular branch groups}
Let $G\le\mathfrak{A}$ be a level-transitive self-similar closed group. Then the following are all equivalent:
\begin{enumerate}[\normalfont(i)]
    \item The group $G$ is regular branch.
    \item The group $G$ is regular branch over its $(M-1)$st level stabilizer for some $M\ge 1$.
    \item There exists $M\ge 1$ such that $s_n=0$ for all $n\ge M$.
    \item The group $G$ is commensurable with $G\times\overset{m}{\ldots}\times G$ via $\psi:\mathrm{St}_G(1)\to \psi(\mathrm{St}_G(1))$.
\end{enumerate}
\end{theorem}

\begin{proof}
Both (ii)$\implies$(i) and (i)$\implies$(iv) are  immediate. Now (iv)$\implies$(iii) follows from $r_n=\log|G\times\overset{m}{\ldots}\times G:\mathfrak{S}_n|$ and $\psi(\mathrm{St}_G(1))$ being of finite index in $G\times\overset{m}{\ldots}\times G$. For (iii)$\implies$(ii) we use that the closure of $\psi(\mathrm{St}_G(1))$ in the direct product $G\times\overset{m}{\ldots}\times G$ is $\bigcap_{n\ge 1} \mathfrak{S}_n=\mathfrak{S}_M$. As $\psi(\mathrm{St}_G(1))$ is closed we get $\psi(\mathrm{St}_G(1))=\mathfrak{S}_M$. Therefore the result follows from $\psi(\mathrm{St}_G(M))=\psi(\mathrm{St}_G(1))\cap\big( \mathrm{St}_G(M-1)\times\overset{m}{\ldots}\times \mathrm{St}_G(M-1)\big)$.
\end{proof}

We can use \cref{theorem: characterization of profinite regular branch groups} to obtain another characterization of regular branchness in terms of virtual fractality and  branchness. We need the following straightforward lemma.

\begin{lemma}
\label{lemma: projections of fractal rist are of finite index}
Let $G$ be a weakly self-similar branch group. Then $\psi_v(\mathrm{rist}_G(v))$
has finite index in $G$ for every vertex $v$ at the first level of the tree if and only if  $G$ is `virtually' fractal. 
\end{lemma}
\begin{proof}
Let us assume $G$ is `virtually' fractal and let us fix a vertex $v$ at the first level of the tree. By assumption, the map $\psi_v:\mathrm{st}_G(v)\to G$ has finite-index image in $G$. Since $G$ is branch we get that $\psi_v(\mathrm{rist}_G(v))=\psi_v(\mathrm{Rist}_G(1))$ is of finite index in $\psi_v(\mathrm{st}_G(v))$ and therefore $\psi_v(\mathrm{rist}_G(v))$ has finite index in $G$. The converse is immediate.
\end{proof}

\begin{theorem}
\label{theorem: for fractal profinite branch is equivalent to regular branch}
Let $G\le\mathfrak{A}$ be a closed self-similar group. Then the group $G$ is regular branch if and only if $G$ is `virtually' fractal and branch.
\end{theorem}
\begin{proof}
Regular branch groups are `virtually' fractal and branch immediately from the definition so we only need to prove the converse. Therefore let us assume that the group $G$ is `virtually' fractal and branch. Then from \cref{lemma: projections of fractal rist are of finite index} we get that
$$\psi(\mathrm{Rist}_G(1))=\psi_{v_1}(\mathrm{rist}_G(v_1))\times\dotsb\times \psi_{v_m}(\mathrm{rist}_G(v_m))$$
has finite index in $G\times\overset{m}{\dotsb}\times G$ where $v_1,\dotsc,v_m$ denote the vertices at the first level of the tree from left to right. The result follows from $(\mathrm{iv})\implies(\mathrm{i})$ in \cref{theorem: characterization of profinite regular branch groups}.
\end{proof}

Virtual fractality cannot be dropped from \cref{theorem: for fractal profinite branch is equivalent to regular branch}; see \cref{Theorem: non-regular self-similar branch groups} for a source of counterexamples.

Lastly, since \cref{Conjecture: New horizons} reduces to just infinite branch profinite groups having positive Hausdorff dimension, and regular branch profinite groups have positive Hausdorff dimension by Bartholdi's result \cite[Proposition 2.7]{BartholdiHausdorff}, the equivalence in \cref{theorem: for fractal profinite branch is equivalent to regular branch} yields a positive result on \cref{Conjecture: New horizons}.

\begin{corollary}
\label{Corollary: Conjecture 1 holds for weakly fractal}
\cref{Conjecture: New horizons} holds for self-similar `virtually' fractal just infinite pro-$p$ groups.
\end{corollary}

Self-similarity may be dropped from \cref{Corollary: Conjecture 1 holds for weakly fractal} as virtual fractality and branchness are sufficient to produce a branching subgroup as seen in the proof of \cref{theorem: characterization of profinite regular branch groups}, and this already yields positivity of the Hausdorff dimension; see \cref{proposition: weakly branch Hausdorff}.

\subsection{Restrictions on the spectrum}

For $d\ge 1$ we say $P\le \mathrm{Sym}(m)\wr\overset{d}{\dotsb}\wr \mathrm{Sym}(m)$ is a \textit{pattern of depth $d$}. Then for such a pattern $P$ the corresponding group of \textit{finite type} $F_{P}\le\mathfrak{A}$ is the group such that an automorphism $g\in \mathfrak{A}$ belongs to $F_{P}$ if and only if $g|_v^d$ coincides with an element in $P$ for every vertex $v$ of the $m$-adic tree. A well-known result of Grigorchuk \cite[Proposition 7.5]{GrigorchukFinite} and \v{S}uni\'{c} \cite[Theorem 3]{SunicHausdorff} says that closed regular branch groups branching over their $(d-1)$st level stabilizer are precisely the level-transitive groups of finite type given by a pattern of size $d$. It follows from \cite[Lemma 1.2.1]{SiegenthalerPhD} that for a closed self-similar group $G$ we obtain a filtration 
\begin{align*}
    W_{G_1}\ge F_{G_2}\ge \dotsb \ge \bigcap_{n\ge 1} F_{G_n}=G
\end{align*}
where each $F_{G_n}$ is the group of finite type given by the pattern of depth $n$ corresponding to the $n$th congruence quotient $G_n$. Then 
\cref{theorem: Formula Hausdorff dimension} together with \cref{proposition: self-similar and branching formula} and \cref{theorem: characterization of profinite regular branch groups} yield that 
$$\lim_{n\to\infty}\mathrm{hdim}_{W_{G_1}}(F_{G_n})=1-\frac{1}{\log|G_1|}\lim_{n\to\infty}\sum_{i=1}^n\frac{s_n}{m^n}=\mathrm{hdim}_{W_{G_1}}(G).$$
This formula can be used to give an approximation of the Hausdorff dimension of a closed self-similar group from above. It also restricts the values the Hausdorff dimension of a closed self-similar group can attain.

\begin{corollary}
\label{Corollary: denseness of spectra}
    For any $H\le\mathrm{Sym}(m)$, the finite type Hausdorff spectrum of $W_H$ is dense in the self-similar Hausdorff spectrum of $W_H$ with respect to the standard topology in $[0,1]$. In particular, the self-similar Hausdorff spectrum of $W_H$ has a gap if and only if its finite type Hausdorff spectrum has a gap.
\end{corollary}

\subsection{An application to a question of Bartholdi}

We prove that if a closed self-similar subgroup $G\le W_H$ has Hausdorff dimension 1 in $W_H$ then it must be equal to $W_H$. Then since Bondarenko proved that $W_H$ is topologically finitely generated if and only if $H$ is perfect \cite[Corollary~3]{Bondarenko1}, \cref{theorem: fg spectrum} follows.

\begin{proof}[Proof of \cref{theorem: fg spectrum}]
    Since $G$ is self-similar all the terms of the sequence $\{s_n\}_{n\ge 1}$ are non-negative, so by the formula in \cref{theorem: Formula Hausdorff dimension} the group $G$ has Hausdorff dimension~1 in $W_H$ if and only if $s_n=0$ for all $n\ge 1$ and $\log|G_1|=\log|H|$. Since $r_1=0$ this implies $r_n=0$ for all $n\ge 1$ and we get
    \begin{align}
    \label{align: proof of fg spectrum}
        \psi(\mathrm{St}_{G}(1))=G\times\overset{m}{\dotsb}\times G
    \end{align}
    by \cref{align: rn and sn self-similar}.
   
    It is enough to prove that $G$ splits over its first level stabilizer. Indeed $G_1=H$ since $G_1\le H$ and they both have the same orders. Therefore, if $G$ splits over its first level stabilizer then $G=W_H$. Let $g\in G$ be any element in $G$ and let us write $\psi(g)=(g_1,\dotsc,g_m)g|_\emptyset^1$. From \textcolor{teal}{(}\ref{align: proof of fg spectrum}\textcolor{teal}{)} there exists $\widetilde{g}\in \mathrm{St}_G(1)$ such that $\psi(\widetilde{g})=(g_1,\dotsc,g_m)$. Therefore
    $$\psi(g\widetilde{g}^{-1})=(1,\dotsc,1)g|_\emptyset^1$$
    and $g\widetilde{g}^{-1}\in G$ proving $G$ splits over its first level stabilizer.
\end{proof}

The above proof shows that for a closed self-similar group $G\le W_H$ the following are all equivalent:
\begin{enumerate}[\normalfont(i)]
    \item the group $G$ has Hausdorff dimension 1 in $W_H$;
    \item the group $G$ has trivial sequences $r_n=s_n=0$ for all $n\ge 1$;
    \item we have the equality $G=W_H$;
    \item the group $G$ is non-trivial and branching.
\end{enumerate}
The equivalence (iii)$\iff$(iv) was proved by Siegenthaler in \cite[Proposition 1.9.1]{SiegenthalerPhD}. 

\subsection{Interpretation of \cref{theorem: Formula Hausdorff dimension}}
Before concluding the section, we motivate the choice of stating \cref{theorem: Formula Hausdorff dimension} in terms of the generating function of the sequence $\{s_n\}_{n\ge 1}$ instead of the one of $\{r_n\}_{n\ge 1}$. Our motivation is twofold. On the one hand, the sequence $s_n$ is in many cases bounded by $0\le s_n\le m-1$. Thus the formula in \cref{theorem: Formula Hausdorff dimension} yields an $m$-base expansion of the Hausdorff dimension. Indeed, let $G\le W_H$ be a closed self-similar group where $\log|G_1|=\log|H|=1$ then 
$$1-\mathrm{hdim}_{W_H}(G)=S_G(1/m)=\sum_{n=1}^\infty \frac{s_n}{m^n}.$$
Under the  assumption $0\le s_n\le m-1$, evaluating the generating function $S_G(x)$ at $x=1/m$ yields the $m$-base expansion of the number $1-\mathrm{hdim}_{W_H}(G)\in [0,1]$. These $m$-base expansions are very effective in inductively producing closed subgroups of prescribed Hausdorff dimension; see for instance \cite{KlopschPhD}. However the sequence $\{r_n\}_{n\ge 1}$ is not in general bounded and does not yield such an $m$-base expansion. On the other hand, the sequence $\{s_n\}_{n\ge 1}$ can only ``see$"$ the labels at the $n$th level. This means that if we can construct a group $G$ having consecutive quotients $\mathrm{St}_G(n)/\mathrm{St}_G(n+1)$ of prescribed order we obtain a prescribed sequence $\{s_n\}_{n\ge 1}$. Piecing everything together, we obtain a very powerful tool to study the Hausdorff dimension of self-similar groups. We shall exploit this tool in \cref{section: generalized iterated wreath products} by constructing a family of closed groups generalizing iterated permutational wreath products, tailor-made for this purpose.

\section{A Generalization of iterated permutational wreath products}
\label{section: generalized iterated wreath products}

Let us choose inductively a sequence of subgroups $\mathcal{S}:=\{S_n\}_{n\ge 0}$ such that each $S_n\le \mathrm{St}(n)$ consists of finitary automorphisms of depth $n+1$. In other words, for every $s\in S_n$ we have $s|_v^1=1$ for every vertex $v$ not at level $n$. Note that there is a group isomorphism between the subgroups of $\mathrm{St}(n)$ consisting of finitary automorphisms of depth $n+1$ and the subgroups of $\mathrm{Sym}(m)\times \overset{m^n}{\ldots}\times \mathrm{Sym}(m)$. We want to impose an additional compatibility condition on the sequence of subgroups~$\mathcal{S}$. In order to do so we define a second sequence of subgroups $\{A_n\}_{n\ge 0}$. 

First we choose the subgroup $S_0\le \mathfrak{A}$ without any further restrictions and we set $A_0:=S_0$. Now for every $n\ge 1$ we choose the subgroup $S_n$ such that $S_n$ is normalized by the subgroup $A_{n-1}$, i.e. $A_{n-1}\le N_{\mathfrak{A}}(S_{n})$. Then we define the subgroup $A_n:=A_{n-1}S_n$, which, by construction, is the internal semidirect product $A_n=S_n\rtimes A_{n-1}$.

Lastly, we construct a profinite group $G_\mathcal{S}\le\mathfrak{A}$ acting on the $m$-adic tree as the closure of $\langle A_n~|~n\ge 0\rangle$ in $\mathfrak{A}$, i.e. as the inverse limit
$$G_\mathcal{S}:=\varprojlim_{n\to\infty}A_n,$$
where the connection homomorphisms are the natural quotient maps $A_n\twoheadrightarrow A_{n-1}$. The definition of the group $G_\mathcal{S}$ depends solely on the defining sequence $\mathcal{S}$ as we have $A_n=S_n\rtimes\dotsb\rtimes S_0$, which justifies the notation. One can look at the groups~$G_\mathcal{S}$ as the \textit{iterated permutational semidirect products} $G_\mathcal{S}=\dotsb\rtimes S_1\rtimes S_0$, generalizing iterated permutational wreath products.

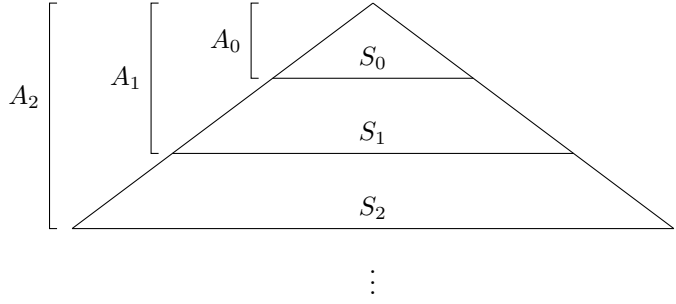
\begin{figure}[H]
    \centering
\begin{tikzpicture}
\coordinate (A) at (-4,0) {};
\coordinate (B) at ( 4,0) {};
\coordinate (C) at (0,3) {};
\draw[name path=AC] (A) -- (C);
\draw[name path=BC] (B) -- (C);
\foreach \y/\A in {0/$S_2$,1/$S_1$,2/$S_0$} {
    \path[name path=horiz] (A|-0,\y) -- (B|-0,\y);
    \draw[name intersections={of=AC and horiz,by=P},
          name intersections={of=BC and horiz,by=Q}] (P) -- (Q)
        node[midway,above] {\A};
}
\coordinate (1) at (-1.52, 3) {};
\coordinate (2) at (-1.52, 2) {};
\coordinate (11) at (-1.62, 3) {};
\coordinate (22) at (-1.62, 2) {};

\coordinate (a) at (-2.85, 3) {};
\coordinate (b) at (-2.85, 1) {};
\coordinate (aa) at (-2.95, 3) {};
\coordinate (bb) at (-2.95, 1) {};

\coordinate (i) at (-4.2, 3) {};
\coordinate (j) at (-4.2, 0) {};
\coordinate (ii) at (-4.3, 3) {};
\coordinate (jj) at (-4.3, 0) {};

\draw (1)--(11)--(22)--(2);
\draw (a)--(aa)--(bb)--(b);
\draw (i)--(ii)--(jj)--(j);
\node at (-1.62,2.5) [left] {$A_0$};
\node at (-2.95,2) [left] {$A_1$};
\node at (-4.3,1.75) [left] {$A_2$};

\node at (0,-1) [above] {$\vdots$};
\end{tikzpicture}
    \caption{The defining sequence $\mathcal{S}$ for the group $G_\mathcal{S}$.}
    \label{figure: Construction of G_L^S}
\end{figure}

Since $G_\mathcal{S}$ is the topological closure of the subgroup $\langle S_n\mid n\ge 0\rangle$ in the ultrametric space $\mathfrak{A}$, where the metric is induced by the level stabilizers, its elements are no more than infinite products
$$g=\prod_{n\ge 0}g_n,$$
where $g_n\in S_n$ and $g|_v^1=g_n|_v^1$ for every vertex $v$ at the $n$th level of the tree for each $n\ge 0$. In other words $G_\mathcal{S}=\prod_{n\ge 0}S_n$. Furthermore, it is also clear by construction that $\mathrm{St}_{G_\mathcal{S}}(n)=\prod_{k\ge n}S_k$ which hence yields $G_\mathcal{S}/\mathrm{St}_{G_\mathcal{S}}(n)\cong\prod_{0\le k\le n-1}S_k=A_{n-1}$. Indeed $G_\mathcal{S}=\mathrm{St}_{G_\mathcal{S}}(n)\rtimes A_{n-1}$ for every $n\ge 1$.

Thus we get a double viewpoint on $G_\mathcal{S}$: on the one hand, the closed subgroup splitting over its $n$th level stabilizer and such that $\mathrm{St}_{G_\mathcal{S}}(n)/\mathrm{St}_{G_\mathcal{S}}(n+1)\cong S_n$ for all $n\ge 1$; and, on the other hand, the closed subgroup whose elements are infinite products $g=\prod_{n\ge 0}g_n$ with $g_n\in S_n$ for all $n\ge 0$, so that $g|_v^1=g_n|_v^1$ for every vertex $v$ at the $n$th level for every $n\ge 0$.

Note that if for every $g_n\in S_n$ we have $g_n|_v^1\in H\le \mathrm{Sym}(m)$ for every vertex $v$, then $G_\mathcal{S}\le W_H$.

\subsection{First examples}
We set $\psi_0(H)=H$ for any $H\le \mathrm{Sym}(m)$. Then the full automorphism group of the $m$-adic tree $\mathfrak{A}$ corresponds to the sequence 
$$\mathcal{S}:=\big\{\psi_n^{-1}\big(\mathrm{Sym}(m)\times\overset{m^n}{\dotsb}\times \mathrm{Sym}(m)\big)\big\}_{n\ge 0}.$$
More generally, for $H\le\mathrm{Sym}(m)$, the iterated permutational wreath product $W_H$ corresponds to the sequence 
$$\mathcal{S}:=\big\{\psi_n^{-1}\big(H\times\overset{m^n}{\dotsb}\times H\big)\big\}_{n\ge 0}.$$

\subsection{Properties of $G_\mathcal{S}$}
Let us study some general properties of the groups $G_\mathcal{S}$.

\begin{proposition}
\label{proposition: self-similar is of the form G_L^S}
The group $G_\mathcal{S}$ is self-similar (or branching) if and only if we have $\psi(S_n)\le S_{n-1}\times\overset{m}{\ldots}\times S_{n-1}$ (respectively $\psi(S_n)\ge S_{n-1}\times\overset{m}{\ldots}\times S_{n-1}$) for all $n\ge 1$. Moreover $G_\mathcal{S}$ is super strongly fractal if and only if for every vertex $x$ at the first level of the tree and for every $n\ge 1$ we have $\psi_x(S_n)=S_{n-1}$.
\end{proposition}
\begin{proof}
First note that if $G_\mathcal{S}$ is self-similar then $\psi(S_n)\le S_{n-1}\times\overset{m}{\ldots}\times S_{n-1}$. Now for the converse let $g\in G_\mathcal{S}$ and let us consider $g=\prod_{n\ge 0}g_n$ where $g_n\in S_n$ for each $n\ge 0$. By assumption we get
\begin{align}
    \label{align: GS sections}
    \psi_k(S_n)\le \psi_{k-1}(S_{n-1})\times\overset{m}{\dotsb}\times \psi_{k-1}(S_{n-1})\le \dotsb\le S_{n-k}\times\overset{m^k}{\dotsb}\times S_{n-k}
\end{align}
for every $1\le k\le n$. Therefore, if for any $k\ge 1$ we consider the section of $g$ at any vertex $v$ at the $k$th level of the tree and of infinite depth, then 
$$g|_v=\big(\prod_{n\ge 0}g_n\big)|_v=\psi_v\big(\prod_{n\ge k}g_n\big)=\prod_{n\ge k}\psi_v(g_n)$$
where $\psi_v(g_n)\in S_{n-k}$ by \cref{align: GS sections}. Thus $g|_v\in \prod_{n\ge 0}S_n=G_\mathcal{S}$. The proof for the branching case is completely analogous, mutatis mutandis.

Lastly, if $G_\mathcal{S}$ is super strongly fractal then we have $\psi_{x}(S_n)=S_{n-1}$ for every $n\ge 1$ and any vertex $x$ at the first level of the tree. To prove the converse we fix a level $k\ge 1$ and a vertex $v$ at the $k$th level and we consider any $g\in G_\mathcal{S}$ where $g=\prod_{n\ge 0}g_n$ and $g_n\in S_n$ for all $n\ge 0$. Since for every vertex $x$ at the first level we have $\psi_x(S_n)=S_{n-1}$ for every $n\ge 1$, then for each $n\ge 0$ there exists an element $\widetilde{g}_n\in S_{n+k}$ such that $\psi_{v}(\widetilde{g}_n)=g_n$. Therefore the element $\tilde{g}:=\prod_{n\ge 0}\tilde{g}_n\in \prod_{n\ge k}S_n=\mathrm{St}_{G_\mathcal{S}}(k)$ satisfies
$$\widetilde{g}|_{v}=\big(\prod_{n\ge 0}\tilde{g}_n\big)|_{v}=\psi_{v}\big(\prod_{n\ge 0}\tilde{g}_n\big)=\prod_{n\ge 0}\psi_{v}(\tilde{g}_n)=\prod_{n\ge 0}g_n=g.\eqno\qedhere$$

\end{proof}

To study level-transitivity of the group $G_\mathcal{S}$ we give first a general result. We just give a short proof of (i)$\implies$(iii) for completeness as the statement seems to not be recorded anywhere despite its simplicity.

\begin{proposition}
        \label{proposition: level-transitivity in general groups}
        Let $G$ be a group acting on the $m$-adic tree. Then, the following are all equivalent:
        \begin{enumerate}[\normalfont(i)]
            \item the group $G$ acts transitively on the $(n+1)$st level of the tree;
            \item for each $0\le k\le n$, there exists a vertex $v$ at the $k$th level  such that $\psi_v(\mathrm{st}_G(v))$ acts transitively on the set $\{1,\dotsc,m\}$;
            \item for every vertex $v$ at the first $n$ levels of the tree we have that $\psi_v(\mathrm{st}_G(v))$ acts transitively on the set $\{1,\dotsc,m\}$.
        \end{enumerate}
        Therefore, the group $G$ being level-transitive is equivalent to both conditions $\mathrm{(ii)}$ and $\mathrm{(iii)}$ for every level of the tree.
\end{proposition}
\begin{proof}
     The implication (iii)$\implies$(ii) is trivial and the proof of (ii)$\implies$(i) follows from a standard argument, see \cite[Lemma 3]{ClassificationAutomata} for instance. Therefore, we just prove (i)$\implies$(iii). Let $v$ be any vertex at a given level of the tree. Let $G$ act transitively one level below $v$. Then, for any $x,y\in \{1,\dotsc,m\}$ there exists an element $g\in G$ sending $vx$ to $vy$. In particular $g$ fixes $v$ and therefore $g\in \mathrm{st}_G(v)$, proving $\psi_v(\mathrm{st}_G(v))$ acts transitively on the set $\{1,\dotsc, m\}$.
\end{proof}

\cref{proposition: level-transitivity in general groups} yields immediately  a characterization of the level-transitivity of the groups $G_\mathcal{S}$.

\begin{corollary}
\label{corollary: S transitive G transitive}
The group $G_\mathcal{S}$ is level-transitive if and only if for every $n\ge 0$ there exists a vertex $v$ at the $n$th level such that $\psi_v(S_n)$ acts transitively on the set $\{1,\dotsc, m\}$.
\end{corollary}

\begin{proposition}
\label{proposition: Sn direct products GLS branch}
Let $G_\mathcal{S}$ be level-transitive. Then $G_\mathcal{S}$ is branch if and only if for any $l\ge 1$ there exists $k\ge l$ such that $ \mathrm{Rist}_{S_n}(l)=S_n$ for all $n\ge k$.
\end{proposition}
\begin{proof}
    This follows directly from the equality $\mathrm{St}_{G_\mathcal{S}}(k)=\prod_{n\ge k}S_k$. Indeed, if for every $l\ge 1$ there exists such a $k\ge l$ then $\mathrm{Rist}_{G_\mathcal{S}}(l)\ge \mathrm{St}_{G_\mathcal{S}}(k)$ and $G_\mathcal{S}$ is branch. Conversely, if $G_\mathcal{S}$ is branch then for every $l\ge 1$, every $\mathrm{Rist}_{G_\mathcal{S}}(l)$ is open and there exists $k\ge l$ such that $\mathrm{Rist}_{G_\mathcal{S}}(l)\ge \mathrm{St}_{G_\mathcal{S}}(k)$. In particular $\mathrm{Rist}_{G_\mathcal{S}}(l)\ge S_n$ and thus $\mathrm{Rist}_{S_n}(l)=S_n$ for all $n\ge k$.
\end{proof}

\begin{proposition}
\label{proposition: sn for GS}
    If the group $G_\mathcal{S}$ is self-similar then
    $$s_n=\log|S_{n-1}\times \overset{m}{\dotsb}\times S_{n-1}:\psi(S_n)|.$$
\end{proposition}
\begin{proof}
    By construction $\mathrm{St}_{G_\mathcal{S}}(n-1)/\mathrm{St}_{G_\mathcal{S}}(n)\cong S_{n-1}$ for every $n\ge 1$ and therefore
    \begin{align*}
        s_n&=m\log|\mathrm{St}_{G_\mathcal{S}}(n-1):\mathrm{St}_{G_\mathcal{S}}(n)|-\log|\mathrm{St}_{G_\mathcal{S}}(n):\mathrm{St}_{G_\mathcal{S}}(n+1)|\\
        &=m\log|S_{n-1}|-\log|S_n|=\log|S_{n-1}\times\overset{m}{\dotsb}\times S_{n-1}:\psi(S_n)|,
    \end{align*}
    where the last equality above follows from self-similarity by \cref{proposition: self-similar is of the form G_L^S}.
\end{proof}

\section{Restricted Hausdorff spectra}
\label{section: restricted Hausdorff spectra}

In this section we restrict to subgroups of $q$-adic automorphisms, i.e. we shall consider a defining sequence $\mathcal{S}$ such that $G_\mathcal{S}\le \Gamma_q$. The spectra (i)--(iii) in \cref{Theorem: restricted Hausdorff spectra p-adic automorphisms} are based on the construction in \cref{lemma: construction of Sn for the ss and wrb cases}. The key in \cref{lemma: construction of Sn for the ss and wrb cases} is to use the condition $\psi(S_n)\le S_{n-1}\times\overset{m}{\ldots}\times S_{n-1}$ in \cref{proposition: self-similar is of the form G_L^S} together with the properties of finite $p$-groups to construct a sequence $\mathcal{S}$ which yields a group $G_\mathcal{S}$ with a prescribed sequence $\{s_n\}_{n\ge 1}$. For the self-similar branch Hausdorff spectrum the construction is based on similar ideas but the focus is shifted to the branchness condition in \cref{proposition: Sn direct products GLS branch}. 

We denote by $D_k(G):=\{(g,\overset{k}{\dotsc},g)~|~g\in G\}\le G\times\overset{k}{\ldots}\times G$ the $k$-diagonal embedding of any group $G$. For $G=\langle g\rangle$ cyclic we shall just write $D_k(g)$.

\subsection{The self-similar, level-transitive and super strongly fractal Hausdorff spectrum}

A natural way to specify the subgroups $S_n$ of the defining sequence $\mathcal{S}$ is to consider $S_0:=\langle \sigma\rangle$ and choose inductively for $n\ge 1$ a subgroup $S_n$ such that $\psi(S_n)\le S_{n-1}\times \overset{q}{\ldots}\times S_{n-1}$ and $\psi(S_n)$ normal in the semidirect product $(S_{n-1}\times \overset{q}{\ldots}\times S_{n-1})\rtimes A_{n-1}$. A subgroup $S_n$ chosen in this way is clearly normalized by $A_{n-1}$ so $\mathcal{S}$ is well defined and the corresponding group $G_\mathcal{S}$ will be self-similar by \cref{proposition: self-similar is of the form G_L^S}. Furthermore, since we want $0\le s_n\le q-1$ so that the sequence $\{s_n\}_{n\ge q}$ yields a $q$-base expansion of a number $\gamma\in [0,1]$ by \cref{proposition: sn for GS}, we may choose $S_n$ such that $\psi(S_n)\ge (H_{n-1}\times\overset{q}{\ldots}\times H_{n-1})D_{q}(S_{n-1})$, where $H_{n-1}$ is a $q$-index normal subgroup of $S_{n-1}$. The fact that $\psi(S_n)$ is an overgroup of the diagonal subgroup $D_q(S_{n-1})$ for every $n\ge 1$ makes $G_\mathcal{S}$ both super strongly fractal and level-transitive by \cref{proposition: self-similar is of the form G_L^S} and \cref{corollary: S transitive G transitive} respectively. We shall use the properties of abelian-by-cyclic groups and of finite $p$-groups to make this construction.

\begin{lemma}
\label{lemma: commutators in Cq wr Cq}
    Let $W=A\langle x\rangle$ where $A$ is an abelian normal subgroup. Then for any $N\le A$ normal in $W$ we have $|N:[N,W]|=|C_{N}(x)|$. In particular, for $W=\langle \sigma\rangle \wr\langle \sigma\rangle$ and $N\le \langle \sigma\rangle \times\overset{q}{\ldots}\times \langle\sigma\rangle$ normal in $W$ such that $N\ge D_q(\sigma)$, we get  $|N:[N,W]|=q$.
\end{lemma}
\begin{proof}
Since $N\le A$ is abelian, we get a group homomorphism $f:N\to [N,W]$ given by $n\mapsto [n,x]$. As $W$ is abelian-by-cyclic we get $[N,W]=[N,x]$, and the map $f$ is surjective. Note that the kernel of $f$ is just the centralizer of $x$ in $N$ by construction and since $f$ is surjective we obtain $|N:[N,W]|=|\ker f|=|C_{N}(x)|$. For $W=\langle \sigma\rangle \wr\langle \sigma\rangle$ the result follows from the above as $W$ is abelian-by-cyclic and $C_N(\sigma)=D_q(\sigma)$ is of order $q$.
\end{proof}

\begin{lemma}
\label{lemma: construction of Sn for the ss and wrb cases}
Let $\{\mu_n\}_{n\ge 1}$ be a sequence of integers such that $0\le \mu_n\le q-1$ for all $n\ge 1$. Then the sequence $\mathcal{S}$ can be chosen so that for every $n\ge 1$:
\begin{enumerate}[\normalfont(i)]
    \item there is an $A_{n-1}$-invariant subgroup $H_n$ of index $q$ in $S_n$ and such that we have both $H_n\ge [S_n,A_{n-1}]$ and  $\psi(H_n)\ge H_{n-1}\times\overset{q}{\ldots}\times H_{n-1}$;
    \item we have $\big(H_{n-1}\times\overset{q}{\ldots}\times H_{n-1}\big)D_q(S_{n-1}) \le \psi(S_n) \le S_{n-1} \times\overset{q}{\ldots}\times S_{n-1}$ and $\psi(S_n)$ has index $q^{\mu_n}$ in $S_{n-1} \times\overset{q}{\ldots}\times S_{n-1}$. 
\end{enumerate}
What is more, the resulting closed group $G_\mathcal{S}$ is self-similar, super strongly fractal and level-transitive.
\end{lemma}
\begin{proof}
We fix $S_0:=\langle \sigma\rangle$ and $H_0=1$. Now for $n=1$ the normal subgroup $\psi^{-1}(D_q(S_0))$ is normalized by $A_0$ and it has index $q^{q-1}$ in $\psi^{-1}(S_{0}\times \overset{q}{\ldots}\times S_0)$ so by the properties of $p$-groups we may find a chain of normal subgroups of the semidirect product $\psi^{-1}\big(S_{0}\times \overset{q}{\ldots}\times S_0\big)\rtimes A_0$ such that
$$N_{q}:=1< N_{q-1}:=\psi^{-1}(D_q(S_0))<N_{q-2}<\dotsb <N_{0}:=\psi^{-1}(S_0\times\overset{q}{\dotsb}\times S_0)$$
and $|N_{k}:N_{k+1}|=q$ for $k=0,\dots,q-1$. We then set $S_1:=N_{\mu_1}$. By \cref{lemma: commutators in Cq wr Cq} we have $|S_1:[S_{1}, A_0]|=q$ and we may set $H_1:= [S_{1}, A_0]$. This choice of $S_1$ and $H_1$ satisfies both conditions (i) and (ii) for $n=1$.

Now by induction we have a subgroup $S_{n-1}$ normalized by $A_{n-2}$ and a subgroup $H_{n-1}$ satisfying both (i) and~(ii) . Then for the quotient 
$$Q:=\frac{\psi^{-1}(S_{n-1}\times \overset{q}{\dotsb}\times S_{n-1})}{\psi^{-1}(H_{n-1}\times\overset{q}{\dotsb}\times H_{n-1})}$$
we have $[Q,A_{n-1}]=[Q,A_0]$ as $\psi(A_{n-1})\le (A_{n-2}\times\overset{q}{\ldots}\times A_{n-2})A_0$ and the action by conjugation of $A_{n-2}$ on the quotient $S_{n-1}/ H_{n-1}$ is trivial. Thus, the subgroup $\psi^{-1}\big(H_{n-1}\times\overset{q}{\ldots}\times H_{n-1}\big) \psi^{-1}(D_q(S_{n-1}))$ is normalized by $A_{n-1}$. Since this subgroup clearly has index $q^{q-1}$ in $\psi^{-1}(S_{n-1}\times \overset{q}{\ldots}\times S_{n-1})$, the same reasoning as in the case $n=1$ produces a subgroup $N\le Q$ normalized by $A_{n-1}$, of index $q^{\mu_n}$ in $Q$ and such that $|N:[N,A_{n-1}]|=q$ by \cref{lemma: commutators in Cq wr Cq}. We then set $S_n$ and $H_n$ to be the lifts of $N$ and $[N,A_{n-1}]$ to $\psi^{-1}(S_{n-1}\times\overset{q}{\ldots}\times S_{n-1})$ respectively. This choice of $S_n$ and $H_n$ satisfies both (i) and (ii).

The self-similarity and the super strong fractality of the group $G_\mathcal{S}$ follow from \cref{proposition: self-similar is of the form G_L^S}. The level-transitivity of the group $G_\mathcal{S}$ follows by the transitivity of the local action of the group $S_n$ on the set $\{1,\dotsc,q\}$ by \cref{corollary: S transitive G transitive}, since $S_n\ge \psi^{-1}(D_q(S_{n-1}))\ge \dotsb\ge \psi_n^{-1}(D_{q^n}(\sigma))$.
\end{proof}

\begin{proof}[{Proof of \cref{Theorem: restricted Hausdorff spectra p-adic automorphisms}\textcolor{teal}{\normalfont(i)}}]
    Let $\gamma\in [0,1]$. Then the $q$-base expansion of $1-\gamma$ is
    $$1-\gamma=\sum_{n=1}^\infty\frac{\mu_n}{q^n},$$
    where $0\le \mu_n\le q-1$. We obtain a profinite self-similar, super strongly fractal and level-transitive group $G_\mathcal{S}$ by specifying a sequence $\mathcal{S}$ where
    $$|S_{n-1}\times\overset{q}{\dotsb}\times S_{n-1}:\psi(S_n)|=q^{\mu_n}$$
    using \cref{lemma: construction of Sn for the ss and wrb cases}.
    Since $G_\mathcal{S}$ is closed self-similar and $\log|(G_\mathcal{S})_1|=\log|(\Gamma_q)_1|=1$ we may compute its Hausdorff dimension using the formula in \cref{theorem: Formula Hausdorff dimension} as
\begin{align*}
    \mathrm{hdim}_{\Gamma_q}(G_\mathcal{S})=1-S_{G_\mathcal{S}}(1/q)=1-\sum_{n=1}^\infty\frac{s_n}{q^n}
\end{align*}
where
$$s_n=\log|S_{n-1}\times\overset{q}{\dotsb}\times S_{n-1}:\psi(S_n)|=\mu_n$$
by \cref{proposition: sn for GS} yielding
\begin{align*}
    \mathrm{hdim}_{\Gamma_q}(G_\mathcal{S})&=1-\sum_{n=1}^\infty\frac{\mu_n}{q^n}=1-(1-\gamma)=\gamma.\qedhere
\end{align*}
\end{proof}

\subsection{The weakly regular branch Hausdorff spectrum}
We first prove that the Hausdorff dimension of a closed weakly regular branch group acting on the $m$-adic tree must be positive for any $m\ge 2$. This generalizes Bartholdi's corresponding result for closed regular branch groups \cite[Proposition 2.7]{BartholdiHausdorff}.

\begin{proposition}
\label{proposition: weakly branch Hausdorff}
The closure of a non-trivial branching group has positive Hausdorff dimension. In particular, the closure of a weakly regular branch group has positive Hausdorff dimension.
\end{proposition}
\begin{proof}
Let $K$ be a branching group. As the Hausdorff dimension of the closure of a subgroup gives a lower bound for the Hausdorff dimension of the closure of the group the second statement follows from the first one. Then let us prove the first statement. We may assume $K$ is closed as $\psi$ is continuous. Let us assume $K$ has trivial Hausdorff dimension. Then we prove that $K$ is trivial. By \cref{theorem: Formula Hausdorff dimension} and \cref{proposition: self-similar and branching formula} we must have both $K/\mathrm{St}_K(1)=1$ and $r_n=0$ for all $n\ge 1$ as $K$ has trivial Hausdorff dimension. The sequence $\{r_n\}_{n\ge 1}$ being constant and equal to zero is equivalent to $\psi(\mathrm{St}_K(1))=K\times\overset{m}{\ldots}\times K$ by the alternative definition of the sequence $\{r_n\}_{n\ge 1}$ for branching groups. Therefore $K$ must be trivial as its action on the first layer of the tree is trivial.
\end{proof}

\cref{lemma: construction of Sn for the ss and wrb cases} yields a group $G_\mathcal{S}$ with zero Hausdorff dimension in $\Gamma_q$ if and only if $\mu_n=q-1$ for all $n\ge 1$. Therefore, since by \cref{proposition: weakly branch Hausdorff} weakly regular branch closed subgroups have positive Hausdorff dimension in $\Gamma_q$ we may assume that $\mu_n$ is not equal to $q-1$ for every $n\ge 1$.

\begin{lemma}
    \label{lemma: non-triviality of Hn}
    Let $\{\mu_n\}_{n\ge 1}$ be a sequence of integers satisfying the assumptions in \cref{lemma: construction of Sn for the ss and wrb cases}. If there exists some $k\ge 1$ such that $\mu_k\ne q-1$, then the $A_{k-1}$-invariant subgroup $H_k$ in \cref{lemma: construction of Sn for the ss and wrb cases} is non-trivial. What is more,
    $$\psi_{n-k}(S_n)\ge H_k\times\overset{q^{n-k}}{\dotsb}\times H_k\ge D_{q^k}(\sigma)\times\overset{q^{n-k}}{\dotsb}\times D_{q^k}(\sigma)$$
    for all $n>k$.
\end{lemma}
\begin{proof}
    Let $k$ be the smallest integer such that $\mu_k\ne q-1$. Then $S_n:=\psi^{-1}(D_q(S_{n-1}))$ and $H_n=1$ for all $0\le n<k$. For $n=k$, by the proof of \cref{lemma: construction of Sn for the ss and wrb cases} we have the chain of normal subgroups
    $$N_{q}:=1< N_{q-1}:=\psi^{-1}(D_q(S_{k-1}))<N_{q-2}<\dotsb <N_{0}:=\psi^{-1}(S_{k-1}\times\overset{q}{\dotsb}\times S_{k-1})$$
and $|N_{i}:N_{i+1}|=q$ for $i=0,\dots,q-2$. Since by assumption $\mu_k\ne q-1$~we get $S_k:=N_{\mu_k}>N_{q-1}=\psi^{-1}(D_q(S_{k-1}))$ and $H_k:=N_{\mu_k-1}\ge N_{q-1}=\psi^{-1}(D_q(S_{k-1}))$. The last statement follows from both the equality $\psi_{k-1}(D_q(S_{k-1}))=D_{q^k}(\sigma)$ for $k$ minimal such that $\mu_k\ne q-1$ and property (ii) in \cref{lemma: construction of Sn for the ss and wrb cases}.
\end{proof}

\begin{lemma}
\label{lemma: existence of branching subgroup}
    Let $G$ be a group acting on the $m$-adic tree and let us fix an element $g\in G$. Assume further that for any vertex $v$ of the tree there exists $g_v\in \mathrm{rist}_G(v)$ such that $\psi_v(g_v)=g$. Then the subgroup $K:=\langle g_v~|~v\text{ a vertex of the tree} \rangle$ is branching.
\end{lemma}
\begin{proof}
It is enough to prove that for any $x_i=1,\dotsc ,m$ and any vertex $v$ there exists an element $k\in K\cap \mathrm{rist}_G(x_i)$ such that $\psi_{x_i}(k)=g_v$. However, the element $g_{x_iv}\in K\cap \mathrm{rist}_G(x_iv)\le K\cap \mathrm{rist}_G(x_i)$ satisfies the condition $\psi_{x_i}(g_{x_iv})=g_v$ from both $\psi_{x_i}(g_{x_iv})\in\mathrm{rist}_G(v)$ and $\psi_v(\psi_{x_i}(g_{x_iv}))=\psi_{x_iv}(g_{x_iv})=g$ since the map $\psi_v$ is injective when we restrict its domain to $\mathrm{rist}_G(v)$.
\end{proof}

\begin{proof}[Proof of {\cref{Theorem: restricted Hausdorff spectra p-adic automorphisms}\textcolor{teal}{\normalfont(ii)}}]
Let $\gamma\in (0,1]$ and consider 
$$1-\gamma=\sum_{n=1}^\infty\frac{\mu_n}{q^n}$$
the $q$-base expansion of $1-\gamma$, where $0\le\mu_n\le q-1$. Since $\gamma\ne 0$ there must exist an integer $k$ such that $\mu_k\ne  q-1$. Then \cref{lemma: construction of Sn for the ss and wrb cases} yields a closed self-similar and level-transitive group $G_\mathcal{S}$ such that
\begin{align}
\label{align: direct products of C2's}
    \psi_{n-k}(S_n)\ge D_{q^k}(\sigma)\times\overset{q^{n-k}}{\dotsb}\times D_{q^{k}}(\sigma)
\end{align}
for all $n> k$ by \cref{lemma: non-triviality of Hn}.

Let us define the automorphism $h$ by $h|_{u}^1=\sigma$ for every vertex $u$ at the $k$th level of the tree and $h|_w^1=1$ for any other vertex $w$. Then for any vertex~$v$ of the tree there exists $h_v\in \mathrm{rist}_{G_\mathcal{S}}(v)$ such that $\psi_v(h_v)=h$ due to \cref{align: direct products of C2's}. Then the result follows from \cref{lemma: existence of branching subgroup}.
\end{proof}

\subsection{The regular branch Hausdorff spectrum} Bartholdi proved that closed regular branch subgroups of $p$-adic automorphisms have rational Hausdorff dimension \cite[Proposition 2.7]{BartholdiHausdorff}. What is more by \v{S}uni\'{c}'s formula in \cite[Theorem 4]{SunicHausdorff} it is immediate that  the regular branch Hausdorff spectrum of the group of $p$-adic automorphisms $\Gamma_p$ is contained in $\mathbb{Z}[1/p] \cap (0,1]$. Both arguments also work for the group of $q$-adic automorphisms $\Gamma_q$. Furthermore we have that $\mathbb{Z}[1/q]=\mathbb{Z}[1/p]$. Therefore we just need to construct a closed regular branch group of $q$-adic automorphisms with prescribed Hausdorff dimension in $\mathbb{Z}[1/p] \cap (0,1]$. 

\begin{proof}[Proof of {\cref{Theorem: restricted Hausdorff spectra p-adic automorphisms}\textcolor{teal}{\normalfont(iii)}}]
    From \cref{theorem: characterization of profinite regular branch groups}, closed regular branch groups are those self-similar groups with $s_n=0$ for all $n\ge k$ for some integer $k\ge 1$. Now, if we consider any $\gamma\in \mathbb{Z}[1/p] \cap (0,1]$ we may compute a finite $q$-base expansion of $1-\gamma$ and reasoning as in the self-similar case we obtain a regular branch group $G_\mathcal{S}$ of Hausdorff dimension $\gamma$ in $\Gamma_q$, concluding the proof.
\end{proof}

An immediate consequence of the above reasoning is that the Hausdorff dimension alone does not characterize closed regular branch groups among closed self-similar groups as any number in $\mathbb{Z}[1/p]\cap (0,1]$ admits a $q$-base expansion with infinitely many non-trivial terms. Indeed, let $\gamma\in \mathbb{Z}[1/p]\cap (0,1]$, then we may compute an infinite $q$-base expansion of $\gamma$ as
$$\gamma=\frac{k}{q^r}=\frac{k-1}{q^r}+\sum_{n= r+1}^\infty\frac{q-1}{q^{n}}.$$
Then the closed self-similar group $G_\mathcal{S}$ constructed with the above $q$-base expansion and \cref{lemma: construction of Sn for the ss and wrb cases} is not regular branch by \cref{theorem: characterization of profinite regular branch groups} but its Hausdorff dimension is $k/q^r\in \mathbb{Z}[1/p]\cap(0,1]$. 

\subsection{The self-similar branch Hausdorff spectrum}

Previously we showed that for any sequence $\{\mu_n\}_{n\ge 1}$ corresponding to a $q$-base expansion of a real number $\gamma\in [0,1]$ there exists a self-similar, level-transitive profinite group $G_\mathcal{S}$ such that $s_n=\mu_n$ for all $n\ge 1$. These $q$-base expansions are enough to produce the groups needed to prove parts (i) to (iii) in \cref{Theorem: restricted Hausdorff spectra p-adic automorphisms}. However, for part (iv) we need to work with more general sequences $\{\mu_n\}_{n\ge 1}$. In this subsection we first characterize the sequences $\{\mu_n\}_{n\ge 1}$ that can be realized as the sequence $\{s_n\}_{n\ge 1}$ of a self-similar level-transitive group $G_\mathcal{S}$. Secondly we show a natural way to construct those kind of sequences by shifting such a sequence by another sequence. This shifting of sequences will produce self-similar branch groups $G_\mathcal{S}$ of any prescribed Hausdorff dimension in $\Gamma_q$.

Let $\{\mu_n\}_{n\ge 1}$ be a sequence of non-negative integers and let $\{\lambda_n\}_{n\ge 1}$ be an increasing sequence of natural numbers. We define the \textit{$\{\lambda_n\}$-shift} of the sequence $\{\mu_n\}_{n\ge 1}$ to be the sequence $\{\widetilde{\mu}_n\}$ given by 
$$\widetilde{\mu}_n:=\begin{cases}
    q^{\lambda_k}\mu_{k}&\text{ if }n=k+\lambda_k \text{ for some }k\ge 1,\\
    0&\text{ otherwise}.
\end{cases}$$
Note that the above sequence is well defined by the monotonicity of $\{\lambda_n\}_{n\ge 1}$. We call a sequence $\{\mu_n\}_{n\ge 1}$ of non-negative integers satisfying the condition
\begin{align*}
    \sum_{i=1}^{n}q^{n-i}\mu_{i}\le q^{n}-1
\end{align*}
for every $n\ge 1$ an \textit{almost $q$-expansion}. This notion generalizes the usual $q$-base expansions where $0\le \mu_n\le q-1$ for every $n\ge 1$. Finally, we say a sequence $\{\mu_n\}_{n\ge 1}$ is \textit{$G_\mathcal{S}$-realizable} if the sequence $\{\mu_n\}_{n\ge 1}$ coincides with the sequence $\{s_n(G_\mathcal{S})\}_{n\ge 1}$ for a self-similar, level-transitive group $G_\mathcal{S}$ and \textit{branch $G_\mathcal{S}$-realizable} if furthermore the group $G_\mathcal{S}$ is also branch.

\begin{lemma}
    \label{lemma: Sn in terms of sn}
    Let $S_0=\langle \sigma\rangle$ and $\psi(S_n)\le S_{n-1}\times \overset{q}{\ldots}\times S_{n-1}$ for every $n\ge 1$. Then for every $n\ge 0$ we have
    $$\log|S_n|=q^n-\sum_{i=1}^{n} q^{n-i}s_{i}.$$
\end{lemma}
\begin{proof}
    The result is clear for $n=0$ by assumption. The general case follows by induction on $n$ from the equality
    $$\log|S_n|=q\log|S_{n-1}|-s_n$$
    in \cref{proposition: sn for GS}.
\end{proof}

\begin{proposition}
    \label{proposition: almost q-expansions}
    Let $\{\mu_n\}_{n\ge 1}$ be a sequence of natural numbers. We have the following:
    \begin{enumerate}[\normalfont(i)]
        \item the sequence $\{\mu_n\}_{n\ge 1}$ is $G_\mathcal{S}$-realizable if and only if it is an almost $q$-expansion;
        \item for any increasing sequence of natural numbers $\{\lambda_n\}_{n\ge 1}$, the $\{\lambda_n\}$-shift of an almost $q$-expansion $\{\mu_n\}_{n\ge 1}$ is branch $G_\mathcal{S}$-realizable.
    \end{enumerate}
\end{proposition}
\begin{proof}
    Let us first prove part (i). Note that the group $G_\mathcal{S}$ is self-similar if and only if $\psi(S_n)\le S_{n-1}\times\overset{q}{\ldots}\times S_{n-1}$ by \cref{proposition: self-similar is of the form G_L^S} and it is level-transitive if $S_n\ge D_{q^n}(\sigma)$. Let us assume first $\{\mu_n\}_{n\ge 1}$ is an almost $q$-expansion. By the properties of finite $p$-groups and arguing as in the proof of \cref{lemma: construction of Sn for the ss and wrb cases} we only need to prove that $\mathcal{S}$ can be chosen such that
    $$\mu_n\le \log|S_{n-1}\times\overset{q}{\ldots}\times S_{n-1}:\psi(D_{q^n}(\sigma))|=q\log|S_{n-1}|-1$$
    for every $n\ge 1$.  We fix $S_0:=\langle \sigma\rangle$ and then we observe that
    $$\mu_1\le \log|S_{0}\times\overset{q}{\ldots}\times S_{0}:\psi(D_{q}(\sigma))|=q-1.$$
    Let us assume by induction on $n$ that $s_i=\mu_i$ for every $1\le i\le n-1$. Then by \cref{lemma: Sn in terms of sn} we need to prove that
    $$\mu_n\le q\big(q^{n-1}-\sum_{i=1}^{n-1} q^{n-1-i}s_{i}\big)-1,$$
    or equivalently, that
    $$\sum_{i=1}^{n}q^{n-i}\mu_i\le q^n-1,$$
    which is precisely the condition of $\{\mu_n\}_{n\ge 1}$ being an almost $q$-expansion. The above also proves that $\{s_n\}_{n\ge 1}$ is an almost $q$-expansion if $G_\mathcal{S}$ is self-similar and level-transitive by \cref{lemma: Sn in terms of sn}, as level-transitivity implies that $\log|S_n|\ge 1$, concluding the proof of part (i).

    For part (ii) it is enough to prove that $\widetilde{S}_n$ may be chosen in this case such that $\mathrm{Rist}_{\widetilde{S}_n}(\lambda_k)=\widetilde{S}_n$ for all $n\ge k+\lambda_k$ by \cref{proposition: Sn direct products GLS branch} and $s_n(G_{\widetilde{\mathcal{S}}})=\widetilde{\mu}_n$ for all $n\ge 1$, where $\{\widetilde{\mu}_n\}_{n\ge 1}$ is the $\{\lambda_n\}$-shift of the sequence $\{\mu_n\}_{n\ge 1}$. Since $\{\mu_k\}_{k\ge 1}$ is $G_\mathcal{S}$-realizable, there exists a defining sequence $\mathcal{S}$ such that 
    $$\mu_k=q\log|S_{k-1}|-\log|S_k|$$
    for all $k\ge 1$ and satisfying the conditions of \cref{proposition: self-similar is of the form G_L^S} and \cref{corollary: S transitive G transitive}. Now we define $\widetilde{S}_0:=S_0=\langle \sigma\rangle$ and 
    \begin{align*}
        \widetilde{S}_n:=\begin{cases}
            \psi_{\lambda_k}^{-1}(S_k\times\overset{q^{\lambda_k}}{\dotsb}\times S_k)& \text{ if }n=k+\lambda_k \text{ for some }k\ge 1,\\
            \psi^{-1}(\widetilde{S}_{n-1}\times\overset{q}{\dotsb}\times \widetilde{S}_{n-1})&\text{ otherwise}.
        \end{cases}
    \end{align*}
    The condition in  \cref{corollary: S transitive G transitive} is readily satisfied by $\widetilde{\mathcal{S}}$ and $\mathrm{Rist}_{\widetilde{S}_n}(\lambda_k)=\widetilde{S}_n$ for all $n\ge k+\lambda_k$ by construction as $\{\lambda_k\}_{k\ge 1}$ is increasing. Therefore the group $G_{\widetilde{\mathcal{S}}}$ is branch. We just need to check the condition in \cref{proposition: self-similar is of the form G_L^S} to prove $G_{\widetilde{\mathcal{S}}}$ is also self-similar. For $n\ne k+\lambda_k$ the condition is trivially satisfied by construction. For $n=k+\lambda_k$ it follows from
    \begin{align}
        \label{align: self-similar branch index}
        \psi_{\lambda_k+1}(\widetilde{S}_n)&=\psi(S_k)\times\overset{q^{\lambda_k}}{\dotsb}\times \psi(S_k)\\
        &\le (S_{k-1}\times \overset{q}{\dotsb}\times S_{k-1})\times \overset{q^{\lambda_k}}{\dotsb}\times (S_{k-1}\times \overset{q}{\dotsb}\times S_{k-1})\nonumber \\
        &=\psi_{\lambda_k}(\widetilde{S}_{n-1})\times\overset{q}{\dotsb}\times \psi_{\lambda_k}(\widetilde{S}_{n-1}),\nonumber 
    \end{align}
    as the sequence $\mathcal{S}$ satisfies the condition in \cref{proposition: self-similar is of the form G_L^S} by assumption. From \textcolor{teal}{(}\ref{align: self-similar branch index}\textcolor{teal}{)} it is also clear that $s_n(G_{\widetilde{\mathcal{S}}})=\widetilde{\mu}_n$ for all $n\ge 1$. 
\end{proof}

\begin{proof}[Proof of {\cref{Theorem: restricted Hausdorff spectra p-adic automorphisms}\textcolor{teal}{\normalfont(iv)}}]
Let us fix $\gamma\in [0,1]$ and consider the $q$-base expansion of $1-\gamma$
$$1-\gamma=\sum_{n\ge 1}\frac{\mu_n}{q^n}.$$
Then $\{\mu_n\}_{n\ge 1}$ is $G_\mathcal{S}$-realizable by \cref{proposition: almost q-expansions}\textcolor{teal}{(i)} and for any increasing sequence $\{\lambda_n\}_{n\ge 1}$ the $\{\lambda_n\}$-shift sequence $\{\widetilde{\mu}_n\}_{n\ge 1}$ is branch $G_\mathcal{S}$-realizable, yielding a self-similar branch pro-$p$ group $G_\mathcal{S}$ such that
\begin{align*}
    \mathrm{hdim}_{\Gamma_q}(G_\mathcal{S})&=1-S_{G_\mathcal{S}}(1/q)=1-\sum_{n\ge 1}\frac{\widetilde{\mu}_n}{q^n}=1-\sum_{k\ge 1}\frac{q^{\lambda_k}\mu_k}{q^{k+\lambda_k}}\\
    &=1-(1-\gamma)=\gamma.\qedhere
\end{align*}
\end{proof}

The same construction in \cref{proposition: almost q-expansions} may be used for any $m\ge 2$, not necessarily a prime power, to obtain self-similar closed groups acting on the $m$-adic tree which are branch but not regular branch.

\begin{corollary}
\label{Theorem: non-regular self-similar branch groups}
    For any $m\ge 2$ there exist self-similar branch groups which are not regular branch.
\end{corollary}
\begin{proof}
    For $m$ not a prime power just choose the sequence $\mathcal{S}$ such that $S_0:=\langle \sigma\rangle$ and $S_n:=D_q(S_{n-1})$ for every $n\ge 1$, as diagonal subgroups are always $A_{n-1}$-invariant and satisfy the conditions in \cref{proposition: self-similar is of the form G_L^S} and \cref{corollary: S transitive G transitive}. Thus $G_\mathcal{S}$ is self-similar and level-transitive. Therefore the same construction as in the proof of \cref{proposition: almost q-expansions}\textcolor{teal}{(ii)} yields a self-similar closed branch group $G_{\widetilde{\mathcal{S}}}$ which is not regular branch by \cref{theorem: characterization of profinite regular branch groups} as $s_n(G_{\widetilde{\mathcal{S}}})\ne 0$ for infinitely many $n\ge 1$ by construction.
\end{proof}

We remark that the above self-similar closed branch groups $G_{\widetilde{\mathcal{S}}}$ cannot be `virtually' fractal as otherwise they would be regular branch by \cref{theorem: for fractal profinite branch is equivalent to regular branch}.

An immediate consequence of \cref{Theorem: restricted Hausdorff spectra p-adic automorphisms}\textcolor{teal}{(iv)} is that there are uncountably many distinct closed self-similar branch groups non-conjugate in $\mathfrak{A}$, as the Hausdorff dimension is a conjugation invariant.  On the other hand, there are only countably many closed regular branch groups as each closed regular branch group is given by a finite pattern in the $m$-adic tree \cite[Theorem 3]{SunicHausdorff} and there are only countably many possible finite patterns. This shows the classes of closed self-similar branch groups and closed regular branch groups differ from each other more than anticipated.

\section{Zero-dimensional just infinite branch pro-$p$ groups}
\label{section: branch groups}

In the previous section we provided the first examples of zero-dimensional closed branch groups using the family of groups $G_\mathcal{S}$. These groups are never finitely generated when they are subgroups of $q$-adic automorphisms, as shown below in \cref{proposition: not finitely generated}. Therefore they do not provide an  answer to \cref{Conjecture: New horizons}, cf. \cref{proposition: finite generation of just infinite branch pro p groups}. However, they suggest a way of constructing a counterexample to \cref{Conjecture: New horizons}: we can fit together the generators of each $S_n$ in a single directed automorphism. Using this idea we construct finitely generated discrete subgroups of $q$-adic automorphisms which are just infinite and branch and whose closure is zero-dimensional.

\subsection{The groups $G_\mathcal{S}$ are not just infinite}
We first show that when $G_\mathcal{S}$ is an infinite pro-$p$ group it is never finitely generated.

\begin{proposition}
\label{proposition: not finitely generated}
    If all the subgroups $S_n\in \mathcal{S}$ are $p$-groups then the group $G_\mathcal{S}$ is either finite or not topologically finitely generated.
\end{proposition}
\begin{proof}
    If only finitely many $S_n$ are non-trivial then $G_\mathcal{S}$ is finite. Thus let us assume there is an infinite subsequence $\{S_{n_i}\}_{i\ge 1}$ where each $S_{n_i}$ is non-trivial. It is enough to see that $d(A_{n_i})>d(A_{n_i-1})$ for all $i\ge 1$, where $d(H)$ denotes the minimal number of elements needed to generate $H$. Since $A_{n_i}$ is a finite $p$-group we know that $d(A_{n_i})=d(A_{n_i}/A_{n_i}')$. Furthermore, since $S_{n_i}$ is normal in $A_{n_i}$ we also have $A_{n_i}'=[S_{n_i},A_{n_i}]\rtimes A_{{n_i}-1}'$ and 
    $$\frac{A_{n_i}}{A_{n_i}'}=\frac{S_{n_i}\rtimes A_{n_i-1}}{[S_{n_i},A_{n_i}]\rtimes A_{n_i-1}'}\cong\frac{S_{n_i}}{[S_{n_i},A_{n_i}]}\times\frac{A_{n_i-1}}{A_{n_i-1}'}.$$
    Certainly $[S_{n_i},A_{n_i}]<S_{n_i}$ as $A_{n_i}$ is a non-trivial finite $p$-group and therefore
\begin{align*}
    d(A_{n_i})&=d(A_{n_i}/A_{n_i}')=d(A_{n_i-1}/A_{n_i-1}')+d(S_{n_i}/[S_{n_i},A_{n_i}])\ge d(A_{n_i-1})+1.\qedhere
\end{align*}
\end{proof}

Now we show the pro-$p$ groups $G_\mathcal{S}$ are not just infinite as they are not topologically finitely generated.

\begin{proposition}
\label{proposition: finite generation of just infinite branch pro p groups}
Let $G$ be a just infinite pro-$p$ group. Then $G$ is topologically finitely generated.
\end{proposition}
\begin{proof}
Recall the Frattini subgroup of $G$ is closed and it is given by $\Phi(G)=\overline{G^pG'}$ as $G$ is a pro-$p$ group. Assume for a contradiction that $\Phi(G)=1$. In particular this implies $G'=1$ so $G$ is abelian and thus $G$ is isomorphic to infinitely many copies of $\mathbb{Z}$. Therefore $G$ cannot be just infinite yielding a contradiction. Then $\Phi(G)$ is non-trivial, closed and normal, thus of finite index from just infiniteness. Hence, it is open and equivalently $G$ is topologically finitely generated.
\end{proof}

\subsection{Motivation and construction}
\label{section: just infinite pro p trivial dimension}
Our construction of a just infinite branch subgroup of $q$-adic automorphisms with zero-dimensional closure is based on the generalized Fabrykowski-Gupta groups acting on a regular rooted tree \cite{FG}. The details of its definition are rather technical; thus we first give the intuition behind it. Let us fix $q\ge 5$ and the rooted automorphism $\psi(a):=(1,\dotsc,1)\sigma$  for the remainder of the section.

The generalized Fabrykowski-Gupta group $F$ acting on the $q$-adic tree is the group generated by $a$ and $b$ where $\psi(b):=(a,1,\dotsc,1,b)$.  The recursive definition of the generator $b$ may be described as ``letting $a$ act on the leftmost descendant of the rightmost vertex at every level of the tree$"$. In other words, the only non-trivial sections of $b$ are at the vertices $q\overset{n}{\ldots}q1$ for each $n\ge 1$ and they are all equal to $a$. A natural generalization is to consider the sections of $b$ at these vertices to be finitary automorphisms that locally look like $a$, i.e. finitary automorphism of order $q$ whose non-trivial labels are all equal to the $q$-cycle $\sigma$.

Let $\{l_n\}_{n\ge 1}$ be an increasing sequence of levels of the tree. We shall specify the sequence $\{l_n\}_{n\ge 1}$ later. Given an automorphism $g\in\Gamma_q$ we set $d_0(g):=g$ and for every $i\ge 1$ we define the automorphism $d_i(g)\in \Gamma_q$ as
$$\psi_i(d_i(g))=(g,\dotsc,g).$$

For all $n\ge 1$ we define $b_n$ via
$$\psi_{l_{n}}(b_{n})=(d_0(a),d_1(a),\dotsc,d_{q^{l_{n}-1}-1}(a),1,\dotsc,1,b_{n+1}).$$

\begin{figure}[H]
\begin{center}
\begin{forest}
for tree={circle, fill=black, inner sep=0pt, outer sep=0pt, s sep=4mm,}
[,name=root
[, name=a1  [, name= a11, fill=white, edge=white [, name=a111 , fill=white, edge=white][, name=a112, fill=white, edge=white][, name=a113, fill=white, edge=white]][, name=a12, fill=white, edge=white[, name=a121, fill=white, edge=white][, name=a122, fill=white, edge=white][, name=a123, fill=white, edge=white]][,name=a13, fill=white, edge=white [, name=a131, fill=white, edge=white][, name=a132, fill=white, edge=white][, name=a133, fill=white, edge=white]]]
[, name=a2  [, name= a21 [, name=a211, fill=white, edge=white][, name=a212, fill=white, edge=white][, name=a213, fill=white, edge=white]][, name=a22[, name=a221, fill=white, edge=white][, name=a222, fill=white, edge=white][, name=a223, fill=white, edge=white]][,name=a23 [, name=a231, fill=white, edge=white][, name=a232, fill=white, edge=white][, name=a233, fill=white, edge=white]]]
[, name=a3  ,[, name= dotsfinal, edge=white [, name= a31 [, name=a311][, name=a312][, name=a313]][, name=a32[, name=a321][, name=a322][, name=a323]][,name=a33 [, name=a331][, name=a332][, name=a333]]],]];
\node (l1) [below=0cm of a1]{$\sigma$};
\node (l21) [below=0cm of a21]{$\sigma$};
\node (l22) [below=0cm of a22]{$\sigma$};
\node (ldots) [below=0.01cm of a22]{$\quad\quad\quad\quad\dotsb$};
\node (l22) [below=0cm of a23]{$\sigma$};
\node (l311) [below=0cm of a311]{$\sigma$};
\node (l312) [below=0cm of a312]{$\sigma$};
\node (l313) [below=0cm of a313]{$\sigma$};
\node (l321) [below=0cm of a321]{$\sigma$};
\node (l322) [below=0cm of a322]{$\sigma$};
\node (l323) [below=0cm of a323]{$\sigma$};
\node (l331) [below=0cm of a331]{$\sigma$};
\node (l332) [below=0cm of a332]{$\sigma$};
\node (l333) [below=0cm of a333]{$\sigma$};
\node (puntos1) [right=1.4cm of a2]{$\dotsb$};
\node (puntos2) [below=0cm of a3]{$\vdots$};
\node (puntos3) [right=0.3cm of a32]{$\dotsb$};
\node (puntos4) [left=0.5cm of puntos2]{\reflectbox{$\ddots$}};
\draw[teal, dashed, thick] (a1)--(a21)--(a23)--(a311)--(a333);
\end{forest}
\end{center}
\caption{The staircase part of the portrait of $b_n$.}
\label{fig: sigman definition}
\end{figure}
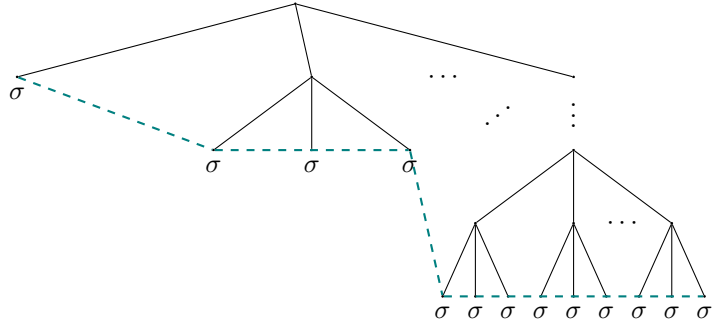

Intuitively the non-trivial labels of $b_n$ can be seen as a staircase: it starts at the leftmost vertex at level $l_{n}$ with $\sigma$ and each step of the staircase is a wider diagonal embedding of $\sigma$'s until level $l_{n}+q^{l_{n}-1}-1$, where it stops and jumps to the portrait given by $b_{n+1}$; see \cref{fig: sigman definition}. Note that for any infinite path starting at the root there is at most one vertex $v$ on this path where the automorphism $b_n$ has a non-trivial label $\sigma$ at~$v$. This ensures that $b_n$ has order $q$, as shown in \cref{lemma: order of an and bn} below.  Furthermore, if we choose the sequence $\{l_n\}_{n\ge 1}$ such that the non-trivial labels of $b_{n+1}$ appear exactly at the level where the labels of $b_n$ become trivial, a group acting transitively on the $l_n$th level and containing $b_n$ will be level-transitive  by \cref{proposition: level-transitivity in general groups}\textcolor{teal}{(ii)}. Therefore, we set $l_1=2$ and $l_{n+1}:=q^{l_n-1}$ for $n\ge 1$ to obtain such an increasing sequence of levels $\{l_n\}_{n\ge 1}$. We fix this sequence $\{l_n\}_{n\ge 1}$ for the remainder of the section.

For every $n\ge 1$ we define the finite group $A_n:=\langle d_i(a)~|~i=0,1,\dotsc, l_n-1\rangle$ of finitary automorphisms. Lastly, for each $n\ge 1$ we define the group $G_n:=\langle A_n,b_{n}\rangle$. Clearly the groups $G_n$ are not self-similar. A word of caution for the reader. The group $G_n$ should not be confused with the $n$th congruence quotient of a group $G$. In order to avoid confusion we shall not use the notation $G_n$ for the congruence quotients and the notation $G_n$ will refer to the group $G_n=\langle A_n,b_n\rangle$ for the remainder of the section.

\subsection{Properties of the groups $G_n$} Now we prove that the groups $G_n$ have the desired properties for every $n\ge 1$.

\begin{lemma}
\label{lemma: order of an and bn}
The elements $d_i(a)$ and $b_n$ have order $q$ for all $i\ge 0$ and $n\ge 1$.
\end{lemma}
\begin{proof}
The elements $d_i(a)$ have clearly order $q$ for all $i\ge 0$. Thus, it follows that 
\begin{align*}
    \psi_{l_{n}}(b_{n}^{\,q})&=(d_0(a)^q, d_1(a)^q,\dotsc,d_{l_{n+1}-1}(a)^q,1,\dotsc,1,b_{n+1}^{\,q})\\
    &=(1,\dotsc,1,{b_{n+1}^{\,q}})=1
\end{align*}
as for any vertex $v$ of the tree, we have $b_{n}^{\,q}|_v^{1}=1$.
\end{proof}

\begin{lemma}
    \label{lemma: An is elementary abelian}
    For every $n\ge 1$, the group $A_n$ is isomorphic to the direct product $C_q\times\dotsb\times C_q$.
\end{lemma}
\begin{proof}
    By \cref{lemma: order of an and bn} the generators of $A_n$ are all of order $q$, and from their portraits it is immediate that $\langle d_i(a)\rangle\cap \langle d_j(a)~|~j\ne i\rangle=1$ for $i=0,1,\dotsc,l_n-1$. Therefore it is enough to show that these generators commute. By definition, it is clear that $\psi(d_i(a))=(d_{i-1}(a),\dotsc,d_{i-1}(a))$, so it suffices to show that $d_i(a)$ commutes with~$a$. This follows from
    $$\psi(d_i(a)^{a})=\sigma^{-1}(d_{i-1}(a),\dotsc,d_{i-1}(a))\sigma=(d_{i-1}(a),\dotsc,d_{i-1}(a))=\psi(d_i(a))$$
    and the injectivity of the map $\psi$.
\end{proof}

\begin{lemma}
    \label{lemma: G is level-transitive}
    For $n\ge 1$, the group $G_n$ is level-transitive.
\end{lemma}
\begin{proof}
The group $A_n$ acts transitively on the first $l_n$ levels of the $q$-adic tree by \cref{proposition: level-transitivity in general groups} as for any level $k<l_n$ there exists $d_k(a)\in \mathrm{St}_{A_n}(k)$ such that $\psi_v(d_k(a))$ acts transitively on the set $\{1,\dotsc,q\}$ for every vertex $v$ at the $k$th level. In particular $\mathrm{St}_{A_n}(1)$ acts transitively at the first $l_n-1$ levels of the tree hanging from any vertex at the first level of the $q$-adic tree. Now for any other level $k\ge l_n$ of the tree there exists a vertex $v$ at the $k$th level such that $b_{n}\in\mathrm{st}_{G_n}(v)$ and $b_{n}|_v^1=a$ by the choice of the sequence $\{l_n\}_{n\ge 1}$, so the result follows from \cref{proposition: level-transitivity in general groups}.
\end{proof}

\begin{proposition}
    \label{proposition: properties of Gn}
    For $n\ge 1$, the group $G_n$ satisfies the following properties:
    \begin{enumerate}[\normalfont(i)]
        \item we have $G_n=\mathrm{St}_{G_n}(l_{n})\rtimes A_n=\langle b_{n}\rangle^{G_n}\rtimes A_n$;
        \item we have $\psi_v(\mathrm{St}_{G_{n}}(l_{n}))=G_{n+1}$ for any vertex $v$ at the $l_{n}$th level of the tree. In particular $\psi_{l_{n}}(\mathrm{St}_{G_{n}}(l_{n}))\le G_{n+1}\times\dotsb\times G_{n+1}$ is a subdirect product;
        \item we have $G_n'=\langle [b_n,d_i(a)]~|~i=0,1,\dotsc,l_n-1\rangle^{G_n}$.
    \end{enumerate}
\end{proposition}
\begin{proof}
Part (i) follows from $A_n$ being a subgroup of finitary automorphisms of depth $l_{n}-1$. Indeed $\mathrm{St}_{\Gamma_q}(l_{n})$ is invariant under the action by conjugation of $A_n$ and since $b_{n}\in \mathrm{St}_{G_n}(l_{n})=\mathrm{St}_{\Gamma_q}(l_{n})\cap G_n$ by construction, the result follows. Part~(ii) is a direct consequence of part (i), of the level-transitivity of $G_n$ and of the recursive definition of $b_n$. Lastly, part (iii) follows directly from \cref{lemma: An is elementary abelian}.
\end{proof}

To simplify notation we introduce the following \textit{block notation}. For every $n\ge 1$ we define the automorphisms $\mathbf{a_{n+1}}$ and $\mathbf{b_{n+1}}$ via
$$\psi_{l_{n}-1}(\mathbf{a_{n+1}})=(d_0(a),\dotsc,d_{l_{n+1}-1}(a))\text{ and }\psi_{l_{n}-1}(\mathbf{b_{n+1}})=(1,\dotsc,1,b_{n+1}).$$
Therefore the definition of $b_n$ looks as follows in block notation
$$\psi(b_{n})=(\mathbf{a_{n+1}},1,\dotsc,1,\mathbf{b_{n+1}}),$$
where $\mathbf{a_{n+1}}$ and $\mathbf{b_{n+1}}$ can be regarded as the main blocks conforming the automorphism $b_n$. Note that $[\mathbf{b_{n+1}},\mathbf{b}^{h}_{\mathbf{n+1}}]=1$ for every $h\in A_n$ and every $n\ge 1$.  The block notation suggests that one can use the same approach as one does for self-similar automata groups, such as the generalized Fabrykowski-Gupta groups, to obtain results on the branchness and just infiniteness of the groups $G_n$. Before we do this, we introduce a last piece of notation. For an automorphism $g\in A_n$ the block $\mathbf{g_*}$ is defined as
$$\psi_{l_{n}-1}(\mathbf{g_{*}})=(*,\dotsc,*,g),$$
where each $*$ represents a (possibly distinct) indeterminate automorphism in $A_{n+1}$. The block $\mathbf{g_*}$ is clearly not uniquely determined, but it provides a useful notation for the proof of \cref{lemma: G is branch}. First we need a general result on branch structures for level-transitive groups. The proof is essentially the same as the one in \cite[Proposition 2.18]{GGSGustavo}.

\begin{lemma}
    \label{lemma: GGS Amaia}
    Let $G$ be a level-transitive group and let $N$ be a normal subgroup of~$G$. Assume further that $\psi_n(\mathrm{St}_G(n))$ is a subdirect product of $H\times\overset{q^n}{\ldots}\times H$ for some group $H$. Let $S$ be a subset of $H$ and set $L=\langle S\rangle^H$. If for some $1\le j\le q^n$ we have
    $$1\times\dotsb\times 1\times S\times 1\times\overset{j}{\dotsb}\times 1\subseteq \psi_n(\mathrm{St}_G(n))$$
    then $L\times\overset{q^n}{\ldots}\times L\le \psi_n(\mathrm{St}_G(n))$.
\end{lemma}

\begin{lemma}
    \label{lemma: G is branch}
   For $n\ge 1$, we have
   $$\psi_{l_{n}}(G_n'')\ge G_{n+1}'\times\overset{q^{l_{n}}}{\dotsb}\times G_{n+1}'.$$
\end{lemma}
\begin{proof}
    Let us fix $n\ge 1$ and then $2\le j\le q-3$ as $q\ge 5$. Recall that we act on the tree on the right. A simple computation shows that
    $$\psi([a^{-j},b_{n}^{-1}])=(\mathbf{a_{n+1}^{\mathrm{-1}}},1,\dotsc,1,\mathbf{b_{n+1}},\mathbf{a_{n+1}},1,\overset{j-2}{\dotsc},1,\mathbf{b_{n+1}^{\mathrm{-1}}}).$$
    
    By \cref{proposition: properties of Gn}\textcolor{teal}{(iii)} we have  $G_n'=\langle [b_n,d_i(a)]~|~i=0,1,\dotsc,l_n-1\rangle^{G_n}$. Therefore by \cref{lemma: GGS Amaia}, we only need to find an automorphism $g\in G_n''$ such that $\psi_{l_{n}}(g)=(1,\dotsc,1,[b_{n+1},d_i(a)],1,\overset{(j+1)q^{l_n-1}}{\ldots},1)$
    for each $i=0,1,\dotsc,l_{n+1}-1$. Let us fix this $i$. Now by the proof of \cref{lemma: G is level-transitive} there exists an element $h\in \mathrm{St}_{A_{n}}(1)$ such that $\mathbf{a}^{h_1}_{\mathbf{n+1}}=\mathbf{d_i(a)_*}$ where $\psi(h)=(h_1,\dotsc,h_q)$. Then
    $$\psi(b_n^h)=(\mathbf{d_i(a)_{*}},1,\dotsc,1,(\mathbf{b_{n+1}})^{h_q}).$$
    Therefore
    $$\psi([a^{-j-1},(b_{n}^{-1})^h])=(\mathbf{d_i(a)_{*}^{\mathrm{-1}}},1,\dotsc,1,(\mathbf{b_{n+1}})^{h_q},\mathbf{d_i(a)_{*}},1,\overset{j-1}{\dotsc},1,(\mathbf{b_{n+1}^{\mathrm{-1}}})^{h_q}).$$
The equality $[\mathbf{b_{n+1}^{\mathrm{-1}}},(\mathbf{b_{n+1}^{\mathrm{-1}}})^{h_q}]=1$ is clear taking into account that $h_q\in A_{n}$; and the equality $[\mathbf{a_{n+1}^{\mathrm{-1}}},\mathbf{d_i(a)_{*}^{\mathrm{-1}}}]=1$ follows from \cref{lemma: An is elementary abelian}. Thus we obtain 
    $$\psi\big(\big[[a^{-j},b_{n}^{-1}],[a^{-j-1},b_{n}^h]\big]\big)=(1,\dotsc,1,[\mathbf{b_{n+1}},\mathbf{d_i(a)_{*}}],1,\overset{j}{\dotsc},1).$$
    Since
    $$\psi_{l_{n}-1}([\mathbf{b_{n+1}},\mathbf{d_i(a)_{*}}])=(1,\dotsc,1,[b_{n+1},d_i(a)]),$$
    the result follows.
\end{proof}

\begin{lemma}
    \label{lemma: Gn has finite abelianization}
    For $n,k\ge 1$, the derived subgroup $G_n^{(k)}$ is of finite index in $G_n$.
\end{lemma}
\begin{proof}
    Note by \cref{lemma: order of an and bn} that for every $n\ge 1$ the group $G_n$ is finitely generated by torsion elements, so the abelianization $G_n/G_n'$ is clearly finite. Assume by induction on $k$ that $G_n^{(k-1)}$ is of finite index in $G_n$ for every $n\ge 1$. Then by \cref{lemma: G is branch},
    $$\psi_{l_{n}}(G_n^{(k)})= \psi_{l_{n}}((G_n'')^{(k-2)}) \ge \big(G_{n+1}'\times\overset{q^{l_{n}}}{\dotsb}\times G_{n+1}'\big)^{(k-2)}=G_{n+1}^{(k-1)}\times\overset{q^{l_{n}}}{\dotsb}\times G_{n+1}^{(k-1)}.$$
    Now $G_{n+1}^{(k-1)}$ having finite index in $G_{n+1}$ implies that $G_n^{(k)}$ is of finite index in $\mathrm{St}_{G_n}(l_{n})$, and thus in $G_n$, by \cref{proposition: properties of Gn} and the injectivity of the homomorphism $\psi_{l_{n}}$.
\end{proof}

To prove just infiniteness we need the following well-known lemma; we refer the reader to \cite[Lemma 4]{pro-c} for a short proof.

\begin{lemma}
    \label{lemma: rist' is in normals}
    Let $G$ be a level-transitive group. Then any non-trivial normal subgroup of $G$ contains the commutator subgroup of some rigid level stabilizer of $G$.
\end{lemma}

\begin{proposition}
\label{proposition: G is just infinite branch}
 For $n\ge 1$, the group $G_n$ is just infinite and branch.
\end{proposition}
\begin{proof}
    Let us fix any $n\ge 1$. By \cref{proposition: properties of Gn}\textcolor{teal}{(ii)} and \cref{lemma: G is branch} all the rigid level stabilizers of $G_n$ are of finite index in $G_n$ and by \cref{lemma: G is level-transitive} the group $G_n$ is level-transitive. Hence $G_n$ is branch. Now for $k\ge 0$ we set $t_k:=\sum_{i=0}^k l_{n+i}$ and observe that by \cref{lemma: G is branch}
    $$\psi_{t_k}(\mathrm{Rist}_{G_n}(t_k)')\ge (G_{n+k+1}'\times\overset{q^{t_k}}{\dotsb}\times G_{n+k+1}')'=G_{n+k+1}''\times\overset{q^{t_k}}{\dotsb}\times G_{n+k+1}''.$$
    Therefore $\mathrm{Rist}_{G_n}(t_k)'$ is of finite index in $G_n$ for all $n\ge 1$ by \cref{lemma: Gn has finite abelianization} and since $t_k\to\infty$ as $k\to\infty$, any non-trivial normal subgroup of $G_n$ contains $\mathrm{Rist}_G(t_k)'$ for some $k\ge 0$ by \cref{lemma: rist' is in normals}. Then it follows that $G$ is just infinite.
\end{proof}

\begin{proposition}
    \label{proposition: G has trivial Hausdorff dimension}
The closure of $G_n$ in $\Gamma_q$ has zero Hausdorff dimension for all $n\ge 1$. 
\end{proposition}
\begin{proof}
Let us set  $i_0:=0$ and also $i_k:=\sum_{j=0}^{k-1}l_{n+j}$ for $k\ge 1$. From \cref{proposition: properties of Gn} we get that $\mathrm{St}_{G_n}(i_k)/\mathrm{St}_{G_n}(i_{k+1})$ embedds in the direct product $A_{n+k}\times \overset{q^{i_k}}{\ldots}\times A_{n+k}$. Since $|A_{n+k}|=q^{l_{n+k}}$ we obtain
$$\log|\mathrm{St}_{G_n}(i_{k}):\mathrm{St}_{G_n}(i_{k+1})|\le q^{i_{k}}l_{n+k}.$$

However
    $$\log|\mathrm{St}_{\Gamma_q}(i_{k}):\mathrm{St}_{\Gamma_q}(i_{k+1})|=q^{i_k}\log|\Gamma_q:\mathrm{St}_{\Gamma_q}(l_{n+k})|=q^{i_{k}}\frac{q^{l_{n+k}}-1}{q-1}.$$
Therefore as $\nabla \{i_{k}\}=\{l_{n+k}\}_{k\ge 0}$ is a non-decreasing sequence we get
\begin{align*}
    \mathrm{hdim}_{\Gamma_q}(\overline{G}_n)&=\liminf_{k\to\infty}\frac{\log|G_n:\mathrm{St}_{G_n}(k)|}{\log|\Gamma_q:\mathrm{St}_{\Gamma_q}(k)|}\\
    &\le\liminf_{k\to\infty}\frac{\log|G_n:\mathrm{St}_{G_n}(i_{k+1})|}{\log|\Gamma_q:\mathrm{St}_{\Gamma_q}(i_{k+1})|}\\
    &\le \liminf_{k\to\infty}(k+1)\frac{\max_{j=0,\dotsc,k}\{\log|\mathrm{St}_{G_n}(i_{j}):\mathrm{St}_{G_n}(i_{j+1})|\}}{\log|\mathrm{St}_{\Gamma_q}(i_{k}):\mathrm{St}_{\Gamma_q}(i_{k+1})|}\\
    &\le(q-1)\lim_{k\to\infty}\frac{(k+1)l_{n+k}}{q^{l_{n+k}}-1}=0. \qedhere
\end{align*}
\end{proof}

The same argument works for any $m\ge 5$ not a prime power too. For $m=2,3,4$ instead of the rooted automorphism $a$ given by the $m$-cycle $\sigma$, one may consider the finitary automorphism of order $m^3$ given by truncating the so-called adding machine $\psi(d)=(1,\dotsc,1,d)\sigma$ at the third level of the tree. Then we set $l_1:=5$ and $l_{n+1}=m^{l_n-3}$ so that we have at least 5 blocks of length at least 3 in the recursive definition of $b_n$. Placing these blocks according to the aforementioned $m^3$-cycle ensures the same arguments as for $m\ge 5$ apply and we obtain a family of groups~$G_n$ with the desired properties acting on the $2,3,4$-adic trees respectively.



\bibliographystyle{unsrt}

\begin{thebibliography}{1}

\bibitem{Abercrombie}
A.\,G. Abercrombie, Subgroups and subrings of profinite rings, \textit{Math. Proc. Camb. Phil. Soc.} \textbf{116 (2)} (1994), 209--222.

\bibitem{AbertVirag}
M. Abért and B. Virág, Dimension and randomness in groups acting on rooted trees, \textit{J. Amer. Math. Soc.} \textbf{18} (2005), 157--192.

\bibitem{BarneaShalev}
Y. Barnea and A. Shalev, Hausdorff dimension, pro-$p$ groups, and Kac-Moody algebras, \textit{Trans. Amer. Math. Soc.} \textbf{349} (1997), 5073--5091.

\bibitem{BarneaMatteo}
Y. Barnea and M. Vannacci, Hereditarily just infinite profinite groups with complete Hausdorff spectrum, \textit{J. Algebra Appl.} \textbf{18 (11)} (2019) 1950216.

\bibitem{BartholdiHausdorff}
L. Bartholdi, Branch rings, thinned rings, tree enveloping rings, \textit{Israel J. Math.} \textbf{158} (2006), 93--139.

\bibitem{FG}
L. Barthodi, B. Eick and R. Hartung, A nilpotent quotient algorithm for $L$-presented groups, \textit{Internat. J. Algebra Compt.} \textbf{18 (8)} (2008), 1321--1344.

\bibitem{Handbook}
L. Bartholdi, R.\,I. Grigorchuk and Z. \v{S}uni\'{c}, Branch groups, in: \textit{Handbook of algebra} \textbf{3}, North-Holland, Amsterdam, 2003.

\bibitem{Holom}
L. Bartholdi and V. Nekrashevych, Thurston equivalence of topological polynomials, \textit{Acta Math.} \textbf{197} (2006), 1--51.

\bibitem{MarialauraBartholdi}
L. Bartholdi and M. Noce, Tree languages and branched groups, \textit{Math. Z.} \textbf{303} (2023), 96.


\bibitem{Bondarenko1}
I.\,V. Bondarenko, Finite generation of iterated wreath products, \textit{Arch. Math.} \textbf{95 (4)} (2010), 301--308.

\bibitem{ClassificationAutomata}
I. Bondarenko, R.\,I. Grigorchuk, R. Kravchenko, Y. Muntyan, V. Nekrashevych, D. Savchuk, and Z. \v{S}uni\'{c}, Classification of groups generated by 3-state automata over a 2-letter alphabet, \textit{Algebra Discrete Math.} \textbf{1} (2008), 1--163.

\bibitem{NewHorizonsBoston}
N. Boston, $p$-adic Galois representations and pro-$p$ Galois groups, in:
\textit{New Horizons in Pro-p Groups}, Birkhäuser Boston, MA \textbf{1}, 2000.

\bibitem{BostonJones}
N. Boston and R. Jones, Arboreal Galois representations, \textit{Geom. Dedicata}, \textbf{124} (2007), 27--35.

\bibitem{BowenF}
L. Bowen, A measure-conjugacy invariant for free group actions, \textit{Ann. of Math.} \textbf{171} (2010), 1387--1400.

\bibitem{BowenMarkov}
L. Bowen, Non-abelian free group actions: Markov processes, the Abramov- Rohlin formula and Yuzvinskii’s formula, \textit{Ergodic Theory Dynam. Systems} \textbf{30(6)} (2010), 1629--1663.

\bibitem{IkerBenjamin}
I. de las Heras and B. Klopsch, A pro-$p$ group with full normal Hausdorff spectra, \textit{Math. Nachr.} \textbf{295 (1)} (2022), 89--102.

\bibitem{IkerAnitha1}
I. de las Heras and A. Thillaisundaram, A pro-2 group with full normal Hausdorff spectra, \textit{J. Group Theory} \textbf{25 (5)} (2022), 867--885.

\bibitem{IkerAnitha2}
I. de las Heras and A. Thillaisundaram, The finitely generated Hausdorff spectra of a family of pro-$p$ groups, \textit{J. Algebra} \textbf{606} (2022), 266--297.

\bibitem{pBasilica}
E. Di Domenico, G.\,A. Fernández-Alcober, M. Noce and A. Thillaisundaram, $p$-Basilica groups, \textit{Mediterr. J. Math.} \textbf{19} (2022), 275.

\bibitem{PadicBook}
J.\,D. Dixon, M.\,P.\,F. du Sautoy, A. Mann, and D. Segal, \textit{Analytic pro-p groups}, Second Edition, Cambridge Studies in Advanced Mathematics, \textbf{61}, Cambridge University Press, Cambridge, 1999.


\bibitem{JorgeMikel}
J. Fariña-Asategui and M.\,E. Garciarena, The Hausdorff dimension of the generalized Brunner-Sidki-Vieira groups, arXiv:2403.16876.

\bibitem{FontaineMazur}
J.-M. Fontaine and B. Mazur, Geometric Galois representations, \textit{Elliptic Curves and Modular Forms}, Proceedings of a conference held in Hong Kong, December 18--21 (1993), International Press, 1997.

\bibitem{GGSGustavo}
G.\,A. Fernández-Alcober and A. Zugadi-Reizabal, GGS-groups: order of congruence quotients and Hausdorff dimension, \textit{Trans. Amer. Math. Soc.} \textbf{366 (4)} (2014), 1993--2017.

\bibitem{OhianaAlejandraBenjamin}
O. Garaialde Ocaña, A. Garrido and B. Klopsch, Pro-$p$ groups of positive rank gradient and Hausdorff dimension, \textit{J. London Math. Soc.} \textbf{101 (3)} (2020), 1008--1040.


\bibitem{pro-c} 
A. Garrido and J. Uria-Albizuri, Pro-$\mathcal{C}$ congruence properties for groups of rooted tree automorphisms, \textit{Arch. Math.} \textbf{112} (2019), 123--137.

\bibitem{AndoniJon}
J. González-Sánchez and A. Zozaya, Standard Hausdorff spectrum of compact $\mathbb{F}_p[t]$-analytic groups, \textit{Monatsh. Math.} \textbf{195 (3)} (2021), 401--419.

\bibitem{NewHorizonsGrigorchuk}
R.\,I. Grigorchuk, Just infinite branch groups, in:
\textit{New Horizons in Pro-p Groups}, Birkhäuser Boston, MA \textbf{1}, 2000.


\bibitem{GrigorchukFinite}
R.\,I. Grigorchuk, Solved and unsolved problems around one group, in: \textit{Infinite groups: Geometric, Combinatorial and Dynamical Aspects}, Progr. Math.
 \textbf{248}, 2005.


\bibitem{KlopschPhD}
B. Klopsch, \textit{Substitution Groups, Subgroup Growth and Other Topics}, D.Phil. Thesis, University of Oxford, 1999.



 \bibitem{PadicAnalytic}
 B. Klopsch, A. Thillaisundaram and A. Zugadi-Reizabal, Hausdorff dimension in $p$-adic analytic groups, \textit{Israel J. Math.} \textbf{231} (2019), 1--23.



 \bibitem{Second}
 M. Noce and A. Thillaisundaram, Hausdorff dimension of the second Grigorchuk group, \textit{Inter. J. Algebra Comput.} \textbf{31 (6)} (2021), 1037--1047.

\bibitem{PenlandMaximalHausdorff}
A. Penland, Nearly maximal Hausdorff dimension in finitely constrained groups, arXiv:1710.05261.

\bibitem{PenlandSunic1}
A. Penland and Z. \v{S}uni\'{c}, A language hierarchy and Kitchens-type theorem for self-similar groups, \textit{J. Algebra} \textbf{537} (2019), 173--196.

\bibitem{PenlandSunic2}
A. Penland and Z. \v{S}uni\'{c}, Finitely constrained groups of maximal Hausdorff dimension, \textit{J. Aust. Math. Soc.} \textbf{100 (1)} (2016), 108--123.


\bibitem{GeneralizedBasilica}
J.\,M. Petschick and K. Rajeev, On the Basilica operation, \textit{Groups Geom. Dyn.} \textbf{17 (1)} (2023), 331--384.

\bibitem{RibesZalesskii}
L. Ribes and P. Zalesskii, Profinite groups, 2nd ed. Springer-Verlag, Berlin-Heidelberg, 2010.


\bibitem{SiegenthalerPhD}
O. Siegenthaler, \textit{Discrete and profinite groups acting on rooted trees}, PhD thesis, University of Göttingen, 2009.

\bibitem{SiegenthalerHausdorff}
O. Siegenthaler, Hausdorff dimension of some groups acting on the binary tree, \textit{Journal of Group Theory} \textbf{11} (2008),  555--567.

\bibitem{SunicHausdorff}
Z. \v{S}uni\'{c}, Hausdorff dimension in a family of self-similar groups, \textit{Geom. Dedicata} \textbf{124} (2007), 213--236.


\bibitem{Wiles}
A. Wiles, Modular Elliptic curves and Fermat's last theorem, \textit{Ann. of Math.} \textbf{141} (1995), 443--551.

\bibitem{Wilson}
J.\,S. Wilson, Groups with every proper quotient finite, \textit{Proc. Camb. Phil. Soc.} \textbf{69} (1971), 373--391.


\bibitem{Zelmanov}
E.\,I. Zelmanov, On periodic compact groups, \textit{Israel J. Math.} \textbf{77} (1992), 83--95.

\end{thebibliography}

\typeout{get arXiv to do 4 passes: Label(s) may have changed. Rerun}

\end{document}